\documentclass{lms}

\usepackage{amsmath,amsbsy,amsfonts,amssymb,calrsfs}
\usepackage{natbib}

\newtheorem{theorem}{Theorem}[section] 
\newtheorem{lemma}[theorem]{Lemma}     

\newtheorem{proposition}[theorem]{Proposition}
\newtheorem{addendum}[theorem]{Addendum}

\newnumbered{assertion}{Assertion}    
\newnumbered{conjecture}{Conjecture}  
\newnumbered{definition}[theorem]{Definition}
\newnumbered{hypothesis}{Hypothesis}
\newnumbered{remark}[theorem]{Remark}
\newnumbered{note}{Note}
\newnumbered{observation}{Observation}
\newnumbered{problem}{Problem}
\newnumbered{question}{Question}
\newnumbered{algorithm}{Algorithm}
\newnumbered{example}[theorem]{Example}

\newunnumbered{notation}{Notation} 
\newunnumbered{conventions}{Conventions}

\newcommand{\eps}{\varepsilon}
\newcommand{\norm}[1]{\| {#1} \|}
\renewcommand{\leq}{\leqslant}
\renewcommand{\geq}{\geqslant}

\newcommand{\bm}{\boldsymbol}
\newcommand{\A}{\boldsymbol{A}}
\newcommand{\B}{{\boldsymbol{B}}}
\newcommand{\C}{{\boldsymbol{C}}}
\newcommand{\D}{{\boldsymbol{D}}}
\renewcommand{\H}{\boldsymbol{H}}
\newcommand{\I}{\boldsymbol{I}}
\newcommand{\J}{\boldsymbol{J}}
\newcommand{\K}{\boldsymbol{K}}
\newcommand{\Q}{{\boldsymbol{Q}}}

\renewcommand{\c}{\boldsymbol{c}}
\newcommand{\g}{\boldsymbol{g}}
\newcommand{\h}{\boldsymbol{h}}
\renewcommand{\v}{\boldsymbol{v}}
\newcommand{\q}{\boldsymbol{q}}
\newcommand{\x}{\boldsymbol{x}}

\newcommand{\0}{\boldsymbol{0}}
\newcommand{\1}{\boldsymbol{1}}
\newcommand{\phib}{\boldsymbol{\phi}}
\newcommand{\etab}{\boldsymbol{\eta}}

\newcommand{\RR}{\mathbb{R}}

\newcommand{\RV}{\mathit{R}}
\newcommand{\GRV}{\mathit{GRV}}

\newcommand{\dt}{\mathrm{d}t}
\newcommand{\du}{\mathrm{d}u}
\newcommand{\dv}{\mathrm{d}v}
\newcommand{\dx}{\mathrm{d}x}
\newcommand{\dy}{\mathrm{d}y}
\newcommand{\dz}{\mathrm{d}z}

\DeclareMathOperator{\diag}{diag}

\title[Generalised regular variation]
 {Generalised regular variation of arbitrary order} 

\author{Edward Omey and Johan Segers}

\classno{26A12 (primary); 39B22, 60G70 (secondary)}

\extraline{This version: \today. The research of the second author was partially supported by IAP research network grant no.\ P6/03 of the Belgian government (Belgian Science Policy).}

\begin{document}
\maketitle

\begin{abstract}
Let $f$ be a measurable, real function defined in a neighbourhood of infinity. The function $f$ is said to be of generalised regular variation if there exist functions $h \not\equiv 0$ and $g > 0$ such that $f(xt) - f(t) = h(x) g(t) + o(g(t))$ as $t \to \infty$ for all $x \in (0, \infty)$. Zooming in on the remainder term $o(g(t))$ leads eventually to a relation of the form $f(xt) - f(t) = h_1(x) g_1(t) + \cdots + h_n(x) g_n(t) + o(g_n(t))$, each $g_i$ being of smaller order than its predecessor $g_{i-1}$. The function $f$ is said to be generalised regularly varying of order $n$ with rate vector $\g = (g_1, \ldots, g_n)'$. Under general assumptions, $\g$ itself must be regularly varying in the sense that $\g(xt) = x^{\B} \g(t) + o(g_n(t))$ for some upper triangular matrix $\B \in \RR^{n \times n}$, and the vector of limit functions $\h = (h_1, \ldots, h_n)$ is of the form $\h(x) = \c \int_1^x u^\B u^{-1} \du$ for some row vector $\c \in \RR^{1 \times n}$. The usual results in the theory of regular variation such as uniform convergence and Potter bounds continue to hold. An interesting special case arises when all the rate functions $g_i$ are slowly varying, yielding $\Pi$-variation of order $n$, the canonical case being that $\B$ is equivalent to a single Jordan block with zero diagonal. The theory is applied to a long list of special functions.

\end{abstract}

\tableofcontents

\section{Introduction} 
\label{S:intro}

The aim of this paper is to provide an analysis of the asymptotic relation
\begin{equation}
\label{E:aim}
	f(xt) = f(t) + \sum_{i=1}^n h_i(x) g_i(t) + o(|g_n(t)|),
	\qquad x \in (0, \infty), \, t \to \infty.
\end{equation}
Here $f$ and $g_1, \ldots, g_n$ are real-valued functions defined on a neighbourhood of infinity; the functions $g_i$ are eventually of constant, non-zero sign and satisfy $g_i(t) = o(|g_{i-1}(t)|)$ as $t \to \infty$ for $i \in \{2, \ldots, n\}$; and $h_1, \ldots, h_n$ are real-valued functions defined on $(0, \infty)$. The functions $g_1, \ldots, g_n$ are called the \emph{rate functions} and the functions $h_1, \ldots, h_n$ are called the \emph{limit functions}.

The following observations are immediate. First, equation~\eqref{E:aim} subsumes the same equation with $n$ replaced by $k \in \{1, \ldots, n\}$: the relation for $k$ can be seen as the one for $k-1$ plus an explicit remainder term. Second, the limit function $h_1$ is determined by $f$ and $g_1$, and the limit function $h_i$ ($i \in \{2, \ldots, n\}$) is determined by $f$ and $g_1, \ldots, g_i$ and $h_1, \ldots, h_{i-1}$.

Equation~\eqref{E:aim} fits into the theory of regular variation in the following way. Recall that a positive, measurable function $F$ defined in a neighbourhood of infinity is called \emph{regularly varying} if for all $x \in (0, \infty)$ the limit $H(x) = \lim_{t \to \infty} F(xt) / F(t)$ exists in $(0, \infty)$ \citep{Karamata30, Karamata33}. If $H \equiv 1$, then $F$ is called \emph{slowly varying}. Setting $f = \log F$ and $h = \log H$ yields $h(t) = \lim_{t \to \infty} \{f(xt) - f(t)\}$, which is case $n = 1$ of \eqref{E:aim} with $g_1 \equiv 1$. If $F$ is measurable, then necessarily $H(x) = x^c$ for some $c \in \RR$, and thus $h(x) = c \log(x)$. The case $n = 1$ of \eqref{E:aim} with general $g_1$ and $h(x) = c \int_1^x u^{b - 1} \du$ was introduced in \citet{dHL70} and is called \emph{extended} or \emph{generalised regular variation}, the symbol $\Pi$ being usually reserved for the class of functions $f$ for which $h(x) = c \log(x)$ with $c \neq 0$. When studying the rate of convergence for $\Pi$-varying functions, one naturally arrives at the case $n = 2$ in \eqref{E:aim}, a case which was studied in \citet{OmeyWillekens88} and \citet{dHS96}. The next step is then $n = 3$, see \citet{FAdHL06}. The jump to arbitrary $n$ was made in \citet{WangCheng05}, where the focus is on the form of the limit functions $h_i$. 

Main references to the theory of regular variation, its extensions and its applications are the monographs by \citet{Seneta76}, \citet{BGT}, and \citet{GdH87}. Besides in the already cited papers, extensions of regular variation are studied in particular in \citet{dH74}, \citet{BinghamGoldie82a, BinghamGoldie82b}, and \citet{GS87}. Regular variation plays an important role in certain areas of probability theory, more precisely in the theory of domains of attraction of stable and max-stable distributions: see \citet{Feller71}, \citet{Resnick87}, and \citet{MS01}. Higher-order extensions are relevant for instance to rates of convergence and Edgeworth expansions related to domains of attractions of max-stable distributions: see \citet{dHR96}, \citet{ChengJiang01}, \citet{WangCheng06}, and \citet{HS07}.

Equation~\eqref{E:aim} is closely related to a similar-looking relation for the auxiliary functions $g_i$: for each $i \in \{1, \ldots, n\}$,
\begin{equation}
\label{E:aim:g}
	g_i(xt) = \sum_{j = i}^n A_{ij}(x) g_j(t) + o(|g_n(t)|),
	\qquad x \in (0, \infty), \, t \to \infty.
\end{equation}
The functions $A_{ij}$, $1 \leq i \leq j \leq n$, are real-valued functions on $(0, \infty)$. They are uniquely determined by the functions $g_1, \ldots, g_n$, and \eqref{E:aim:g} implies the same equation for $n$ replaced by $k \in \{i, \ldots, n\}$.

Our study is then centered around the following questions:
\begin{itemize}
\item If \eqref{E:aim} and \eqref{E:aim:g} are known to be true only for certain subsets of $x$ in $(0, \infty)$, then how large should these subsets be so that \eqref{E:aim} and \eqref{E:aim:g} actually hold for all $x \in (0, \infty)$?
\item What is the relation between \eqref{E:aim} and \eqref{E:aim:g}? In particular, up to what extent does \eqref{E:aim} imply \eqref{E:aim:g}?
\item How do the limit functions $h_i$ and $A_{ij}$ look like? How are they related?
\item Up to what extent are \eqref{E:aim} and \eqref{E:aim:g} (locally) uniform in $x$?
\item Is it possible to give global, non-asymptotic upper bounds (so-called Potter bounds) for the remainder terms in \eqref{E:aim} and \eqref{E:aim:g}?
\item Can we represent $f$ in terms of $g_1, \ldots, g_n$? Can we retrieve $g_1, \ldots, g_n$ from $f$?
\end{itemize}

The key to the answers to these questions is a representation in terms of vectors and matrices. For instance, \eqref{E:aim:g} can be written as
\[
	\g(xt) = \A(x) \g(t) + o(|g_n(t)|), \qquad x \in (0, \infty), \, t \to \infty;
\]
here $\g = (g_1, \ldots, g_n)'$ is a column-vector valued function, and $\A = (A_{ij})_{i,j=1}^n$ is an upper-triangular matrix valued function. The function $\A$ necessarily satisfies the matrix version of the (multiplicative) Cauchy functional equation: $\A(xy) = \A(x) \A(y)$. Under measurability, this equation can be solved, yielding 
\[
	\A(x) = x^{\B}, \qquad x \in (0, \infty),
\]
for some upper triangular matrix $\B \in \RR^{n \times n}$; here $x^{\B} = \exp \{ (\log x) \B \}$, the exponential function of a matrix being defined by the usual series expansion. The rate vector $\g$ will be called \emph{regularly varying} with {index matrix} $\B$; notation $\g \in \RV_{\B}$. The row vector of limit functions $\h = (h_1, \ldots, h_n)$ in \eqref{E:aim} can then be written as
\[
	\h(x) = \bm{c} \int_1^x u^{\B} u^{-1} \du, \qquad x \in (0, \infty),
\]
for some $\bm{c} \in \RR^{1 \times n}$. We say that $f$ is \emph{generalised regularly varying} with rate vector $\g$ and $\g$-index $\c$; notation $f \in \GRV(\g)$.

Conceptually, the $n$th order case is strongly similar to the first-order case provided we are ready to think of the index of regular variation as a matrix. The fact that this matrix is upper triangular comes from the assumption that $g_i(t) = o(|g_{i-1}(t)|)$ and gives rise to some convenient simplifications. The diagonal elements of $\B$ are at the same time the eigenvalues of $\B$ and the indices of regular variation of the functions $|g_i|$, and they determine the nature of \eqref{E:aim}. A rather interesting special case arises when all these indices are equal to zero, yielding functions $f$ which are $\Pi$-varying of order $n$. These functions arise naturally as inverses of $\Gamma$-varying functions of which the auxiliary function has regularly varying derivatives up to some order.

\subsection{Outline}

The paper can be divided into two parts: the core theory in Sections~\ref{S:quantlim}--\ref{S:URP} and special cases and examples in Sections~\ref{S:PiI}--\ref{S:examples}.

In Section~\ref{S:quantlim}, we study the interplay between \eqref{E:aim} and \eqref{E:aim:g} without the assumption of measurability. More precisely, we investigate up to what extent convergence for \emph{some} $x$ implies convergence for \emph{all} $x$, and we determine structural properties of and relations between the functions $h_i$ and $A_{ij}$.

From Section~\ref{S:measurability} on, we assume that $f$ and $g_1, \ldots, g_n$ are measurable. After some preliminaries on matrix exponential functions and powers of triangular matrices (Section~\ref{SS:matrixexp}), we define regularly varying rate vectors (Section~\ref{SS:rvrv}) and generalised regular variation of arbitrary order (Section~\ref{SS:grv}). Main results are the two characterisation theorems, giving the precise form of the limit functions $\A$ and $\h$ in terms of an index matrix $\B$ and an index vector $\c$. 

The next fundamental results are the uniform convergence theorems in Section~\ref{S:URP}, leading to representation theorems and Potter bounds. In these Potter bounds, both the largest and the smallest eigenvalue of the index matrix play a role, depending on whether $x$ is larger or smaller than $1$. 

The special case of $\Pi$-variation of arbirary order is studied in Section~\ref{S:PiI} for general functions and in Section~\ref{S:PiII} for functions with continuous derivatives up to some order. The canonical situation here is that the index matrix $\B$ is equivalent to a single Jordan block with zero diagonal (Section~\ref{SS:Pi:Jordan}). A suitable rate vector can be constructed directly from $f$ through repeated applications of a special case of the indices transform (Section~\ref{SS:Pi:indices}) or by higher-order differencing (Section~\ref{SS:PiI:diff}). For smooth functions, an even simpler procedure is available through repeated differentiation (Section~\ref{SS:PiII:smooth}). This procedure applies in particular to inverses of certain $\Gamma$-varying functions (Section~\ref{SS:PiII:Gamma}). 

Two other special cases are described in Section~\ref{S:speccas}. First, if all eigenvalues of $\B$ are different from zero, then $f$ is essentially just a linear combination of the rate functions $g_1, \ldots, g_n$ (Section~\ref{SS:speccas:nonzero}). Second, if all eigenvalues of $\B$ are distinct, then $f$ is essentially a linear combination of power functions, a logarithm, and the final rate function $g_n$ (Section~\ref{SS:speccas:diff}). 

Finally, the paper is concluded with a long list of examples in Section~\ref{S:examples}. Except for the logarithm of the Gamma function, the most interesting examples all come from the class $\Pi$ of arbitrary order. These include the Lambert $W$ function and the inverses of the reciprocals of the complementary Gamma and complementary error functions.

\subsection{Conventions}

A real function is \emph{defined in a neighbourhood of infinity} if its domain contains an interval of the form $[t_0, \infty)$ for some real $t_0$. A real-valued function $f$ defined in a neighbourhood of infinity is \emph{eventually of constant sign} if there exists a real $t_0$ such that either $f(t) > 0$ for all $t \geq t_0$ or $f(t) < 0$ for all $t \geq t_0$. \emph{Measure} and \emph{measurability} refer to the Lebesgue measure and $\sigma$-algebra. 

Unless specified otherwise, limit relations are to be understood as $t \to \infty$. We write $a(t) \sim b(t)$ if $a(t)/b(t) \to 1$.
 
A column vector $\g = (g_1, \ldots, g_n)'$ of real-valued functions defined in a neighbourhood of infinity is called a \emph{rate vector} if each component function is ultimately of constant sign and, in case $n \geq 2$, if $g_{i+1}(t) = o(g_i(t))$ for every $i \in \{1, \ldots, n-1\}$. Even though the sign of $g_i(t)$ could be eventually negative, for simplicity we will write $o(g_i(t))$ rather than $o(|g_i(t)|)$.

\section{Quantifiers and limits}
\label{S:quantlim}

\subsection{Preliminaries}
\label{SS:prelimin}
For a rate vector $\g$ of length $n$, let $S(\g)$ be the set of all $x \in (0, \infty)$ for which there exists $\A(x) \in \RR^{n \times n}$ such that
\begin{equation}
\label{E:g}
	\g(xt) = \A(x) \g(t) + o(g_n(t)).
\end{equation}
For a real-valued function $f$ defined in a neighbourhood of infinity and a rate vector $\g$ of length $n$, let $T(f, \g)$ be the set of all $x \in (0, \infty)$ for which there exists a row vector $\h(x) = (h_1(x), \ldots, h_n(x))$ such that
\begin{equation}
\label{E:f}
	f(xt) = f(t) + \h(x) \g(t) + o(g_n(t)).
\end{equation}
Whenever clear from the context, we will just write $S = S(\g)$ and $T = T(f, \g)$.

In this section, we will study properties of and relations between the sets $S$ and $T$ and the matrix and vector functions $\A : S \to \RR^{n \times n}$ and $\h : T \to \RR^{1 \times n}$. In particular, we are interested up to what extent \eqref{E:g} is implied by \eqref{E:f}. In Subsection~\ref{SS:g}, we study the properties of $\A$ and $S$, while in Subsection~\ref{SS:f}, we study the interplay between $\A$ and $S$ on the one hand and $\h$ and $T$ on the other hand. We do not assume that $f$ or $\g$ are measurable; this assumption will be added from the next section onwards.

\begin{remark}
Let $f$ be a real-valued function defined in a neighbourhood of infinity and let $\g$ be a rate vector.
\begin{flushenumerate}
\item[(a)] For $x \in T$, the vector $\h(x)$ is uniquely defined by $f$ and $\g$ through the recursive relations
\begin{align*}
	h_{1}(x) &= \lim_{t \to\infty} \{f(tx) - f(t) \} / g_{1}(t), \\
	h_{k}(x) &= \lim_{t \to\infty} 
		\biggl\{ f(tx) - f(t) - \sum_{i=1}^{k-1} h_i(x) g_i(t) \biggr\} \bigg/ g_{k}(t), 
		\qquad k \in \{2, \ldots, n \}.
\end{align*}
\item[(b)] If $\boldsymbol{Q}$ is an invertible, upper triangular $n \times n$ matrix, then $\boldsymbol{Q} \g$ is a rate vector as well and \eqref{E:f} holds true with $\g(t)$ and $\h(x)$ replaced by $\boldsymbol{Q} \g(t)$ and $\h(x) \boldsymbol{Q}^{-1}$, respectively.
\item[(c)] Rather than \eqref{E:f}, one could also study more general relations of the form
\[
	f(tx) = x^{\alpha }f(t) + \h(x)\g(t) + o(g_{n}(t))
\]
for some fixed $\alpha \in \RR$. However, replacing $f(t)$ by $t^{-\alpha }f(t)$, $\g(t)$ by $t^{-\alpha} \g(t)$ and $\h(x)$ by $x^{-\alpha} \h(x)$ would lead back to \eqref{E:f} again.
\end{flushenumerate}
\end{remark}

\subsection{Rate vectors}
\label{SS:g}

We study some elementary properties of the set $S$ and the matrix function $\A$ in \eqref{E:g}.

\begin{proposition}
\label{P:A}
Let $\g$ be a rate vector and let $S = S(\g)$ and $\A(x)$ be as in \eqref{E:g}.
\begin{flushenumerate}
\item[(a)] The matrix $\A(x)$ is upper triangular and is uniquely determined by $\g$ and $x$.
\item[(b)] If $x, y \in S$ then $xy \in S$ and $\A(xy) = \A(x) \A(y)$.
\item[(c)] For $x \in S$, the matrix $\A(x)$ is invertible if and only if $x^{-1} \in S$; in that case, $\A(x)^{-1} = \A(x^{-1})$.
\end{flushenumerate}
\end{proposition}

\begin{proof}
(a) If $n = 1$, then \eqref{E:g} is just $g_1(xt) / g_1(t) \to A_{11}(x)$, so there is nothing to prove. Suppose therefore that $n \geq 2$. Row number $i \in \{1, \ldots, n\}$ in \eqref{E:g} reads
\[
	g_i(xt) = \sum_{j=1}^n A_{ij}(x) g_j(t) + o(g_n(t)), \qquad x \in S.
\]
Since $\g$ is a rate vector, we find
\[
	g_i(xt) / g_1(t) \to A_{i1}(x), \qquad x \in S.
\]
If $i \geq 2$, then also
\[
	\frac{g_i(xt)}{g_1(t)} = \frac{g_i(xt)}{g_1(xt)} \frac{g_1(xt)}{g_1(t)}
	\to 0 \cdot A_{11}(x) = 0, \qquad x \in S.
\]
As a consequence, $A_{i1}(x) = 0$ for $x \in S$ and $i \in \{2, \ldots, n\}$. Therefore, the vector $(g_2, \ldots, g_n)'$, which is a rate vector too, satisfies \eqref{E:g} as well but with $\A(x)$ replaced by $(A_{ij}(x))_{i,j=2}^n$, $x \in S$. Proceed by induction to find that $\A(x)$ is upper triangular. From
\[
	g_i(xt) = \sum_{j=i}^n A_{ij}(x) g_j(t) + o(g_n(t)), \qquad x \in S, i \in \{1, \ldots, n\},
\]
and the fact that $\g$ is a rate vector, it follows that the component functions $A_{ij}$ can be retrieved recursively from $\g$ via the relations
\begin{align}
\label{E:g2A:a}
	A_{ii}(x) &= \lim_{t \to \infty} g_i(xt) / g_i(t), \\
\label{E:g2A:b}
	A_{ij}(x) &= \lim_{t \to \infty} \biggl( g_i(xt) - \sum_{k=i}^{j-1} A_{ik}(x) g_k(t) \biggr) \bigg/ g_j(t)
\end{align}
for $i \in \{1, \ldots, n\}$ and $j \in \{i+1, \ldots, n\}$.

(b) For $x, y \in S$, we have
\begin{align*}
	\g(xyt)
	&= \A(x) \g(yt) + o(g_n(yt)) \\
	&= \A(x) \A(y) \g(t) + o(g_n(t)) + o(g_n(yt)).
\end{align*}
Since $g_n(yt)/g_n(t) \to A_{nn}(y)$, we obtain
\[
	\g(xyt) = \A(x) \A(y) \g(t) + o(g_n(t)),
\]
and thus $xy \in S$. Moreover, $\A(xy) = \A(x)\A(y)$ by uniqueness.

(c) Trivially, $1 \in S$ and $\A(1) = I_n$, the $n \times n$ identity matrix. Therefore, if both $x$ and $x^{-1}$ belong to $S$, then by (b), $\A(x) \A(x^{-1}) = \A(x x^{-1}) = \A(1) = I_n$. Conversely, suppose that $x \in S$ and that $\A(x)$ is invertible. Then
\begin{align*}
	\A(x)^{-1} \g(xt) 
	&= \A(x)^{-1} \{ \A(x) \g(t) + o(g_n(t)) \} \\
	&= \g(t) + o(g_n(t)).
\end{align*}
From the fact that $\A(x)$ is upper triangular and invertible, it follows that $A_{nn}(x)$ is nonzero. But as $g_n(xt) / g_n(t) \to A_{nn}(x)$, we may therefore rewrite the previous display as
\[
	\A(x)^{-1} \g(xt) = \g(x^{-1} (xt)) + o(g_n(xt)).
\]
Putting $s = xt$, we get
\[
	\g(x^{-1}s) = \A(x)^{-1} \g(s) + o(g_n(s)), \qquad s \to \infty.
\]
As a consequence, $x^{-1} \in S$. Moreover, by uniqueness, $\A(x^{-1}) = \A(x)^{-1}$.
\end{proof}

\begin{proposition}
\label{P:S}
Let $\g$ be a rate vector and let $S = S(\g)$.
\begin{flushenumerate}
\item[(a)] If $S \cap (1, \infty)$ contains a set of positive measure, then $(x, \infty) \subset S$ for some $x \in (1, \infty)$. 
\item[(b)] If $S \cap (0, 1)$ contains a set of positive measure, then $(0, x) \subset S$ for some $x \in (0, 1)$. 
\item[(c)] If $S$ contains a set of positive measure and if both $S \cap (0, 1)$ and $S \cap (1, \infty)$ are non-empty, then $S = (0, \infty)$.
\end{flushenumerate}
\end{proposition}

\begin{proof}
By Proposition~\ref{P:A}, $S$ is a multiplicative semigroup. Statements (a) and (b) then follow directly from Corollary~1.1.5 in \citet{BGT}, due to \citet{HillePhillips57}. 

To prove (c), proceed as follows. By assumption, $S$ contains a set of positive measure; hence $S \cap (0, 1)$ or $S \cap (1, \infty)$ must contain a set of positive measure; assume the latter. Take $y \in S \cap (0, 1)$, which is non-empty by assumption. By (a), there exists a positive integer $k$ such that $(y^{-k}, \infty) \subset S$. Since $y \in S$ and since $S$ is a multiplicative semigroup, $(a, \infty) \subset S$ implies $(ay, \infty) \subset S$. By induction, we get that $(y^l, \infty) \subset S$ for every integer $l$. Hence $(0, \infty) \subset S$.
\end{proof}

\subsection{General functions}
\label{SS:f}

We will now study the interplay between equations~\eqref{E:g} and \eqref{E:f}. In Proposition~\ref{P:S2T} we will see how properties of $S$ transfer to properties of $T$. Conversely, in Proposition~\ref{P:T2S}, we investigate up to what extent \eqref{E:f} implies \eqref{E:g}; in particular, we express $\g$ and $\A$ in terms of $f$ and $\h$.

\begin{proposition}
\label{P:S2T}
Let $f$ be a real-valued function defined in a neighbourhood of infinity and let $\g$ be a rate vector. If $S = (0, \infty)$, then $T$ is a multiplicative group and for $x, y \in T$,
\begin{align}
\label{E:hxy}
	\h(xy) &= \h(x) \A(y) + \h(y), \\
\label{E:hx1}
	\h(x^{-1}) &= - \h(x) \A(x)^{-1}.
\end{align}
In particular, if $T$ also contains a set of positive measure, then $T = (0, \infty)$.
\end{proposition}

\begin{proof}
Let $x, y \in T$. We have
\begin{align*}
	f(xyt) - f(t)
	&= \{f(xyt) - f(yt)\} + \{f(yt) - f(t)\} \\
	&= \h(x)\g(yt) + o(g_n(yt)) + \h(y)\g(t) + o(g_n(t)) \\
	&= \h(x)\A(y)\g(t) + \h(y)\g(t) + o(g_n(yt)) + o(g_n(t)) \\
	&= \{\h(x)\A(y) + \h(y)\}\g(t) + o(g_n(t)),
\end{align*}
where in the last step we used the fact that $g_n(yt)/g_n(t) \to A_{nn}(y)$. Hence $xy \in T$, and \eqref{E:hxy} follows from the uniqueness of $\h$.

Next, let $x \in T$. We have
\begin{align*}
	f(x^{-1}t)-f(t)
	&= -\{f(x(x^{-1}t)) - f(x^{-1}t)\} \\
	&= -\h(x)\g(x^{-1}t) + o(g_n(x^{-1}t)) \\
	&= -\h(x)\A(x^{-1})\g(t) + o(g_n(t)) + o(g_n(x^{-1}t)) \\
	&= -\h(x)\A(x^{-1})\g(t) + o(g_n(t)),
\end{align*}
where in the last step we used the fact that $g_n(x^{-1}t)/g_n(t) \to A_{nn}(x^{-1})$. Hence $x^{-1} \in T$, and \eqref{E:hx1} follows from the uniqueness of $\h$.

The final statement follows \citet[Corollary~1.1.4]{BGT}, going back to \citet{Steinhaus20}.
\end{proof}

\begin{remark}
\label{R:C}
Let $f$ be a real-valued function defined in a neighbourhood of infinity and let $\g = (g_1, \ldots, g_n)'$ be a rate vector. Put 
\begin{align}
	\bm{r} &= (f, g_1, \ldots, g_n)', \nonumber \\
\label{E:C}
	\bm{C}(x) &= \begin{pmatrix} 1 & \h(x) \\ \bm{0} & \A(x) \end{pmatrix}, \qquad x \in S \cap T. 
\end{align}
By Proposition~\ref{P:A}(a), $\bm{C}(x)$ is an $(n+1) \times (n+1)$ upper triangular matrix. Then \eqref{E:g} and \eqref{E:f} can be put together as
\[
	\bm{r}(xt) = \bm{C}(x) \bm{r}(t) + o(r_{n+1}(t)),
\]
which is again of the form \eqref{E:g}. If $x, y \in S \cap T$, then by going through the proof of Proposition~\ref{P:S2T} we find that also $xy \in S \cap T$ and that $\bm{C}(xy) = \bm{C}(x) \bm{C}(y)$. Similarly, for $x \in S \cap T$, the matrix $\bm{C}(x)$ is invertible if and only if $x^{-1} \in S \cap T$, in which case $\bm{C}(x^{-1}) = \bm{C}(x)^{-1}$. 

However, $\bm{r}$ is not necessarily a rate vector. Moreover, as we shall be interested in the interplay between \eqref{E:g} and \eqref{E:f}, we continue to study these equations separately.
\end{remark}

Conversely, to go from $T$ and $\h$ to $S$ and $\A$, we need to impose that the functions $h_1, \ldots, h_n$ are linearly independent. For instance, if $n = 1$ and $h_1 \equiv 0$, then \eqref{E:f} reduces to $f(xt) - f(t) = o(g(t))$, from which in general nothing can be deduced about the rate function $g$; likewise, in case $n = 2$, linear independence of the functions $h_1$ and $h_2$ is assumed from the start in \citet{dHS96}. Note that by Remark~\ref{R:linearindependence} below, unless the functions $h_1, \ldots, h_n$ are all identically zero, it is always possible to find a non-empty $I \subset \{1, \ldots, n\}$ and a rate vector $\tilde{\g} = (\tilde{g}_i)_{i \in I}'$ such that the functions $h_i$, $i \in I$ are linearly independent, $\tilde{g}_i(t)/g_i(t) \to 1$ for all $i \in I$, and $f(xt) = f(t) + \sum_{i \in I} h_i(x) \tilde{g}_i(t) + o(g_n(t))$. 

\begin{notation}
Let $f$ and $\h$ be as in \eqref{E:f}. For a vector $\x = (x_1, \ldots, x_n) \in T^n$, put
\begin{equation}
\label{E:hx}
	\h(\x) = (\h(x_1)', \ldots, \h(x_n)')',
\end{equation}
that is, $\h(\x)$ is an $n \times n$-matrix with $\h(x_i)$ as row number $i$. Further, for $\boldsymbol{t} = (t_1, \ldots, t_n)$ such that all $t_i$ are in the domain of $f$, put 
\[
	f(\boldsymbol{t}) = (f(t_1), \ldots, f(t_n))'. 
\]
Finally, put $\1 = (1, \ldots, 1) \in \RR^n$.
\end{notation}

Note that the functions $h_1, \ldots, h_n$ are linearly independent if and only if there exists $\x \in T^n$ such that the vectors $(h_i(x_1), \ldots, h_i(x_n))$, $i \in \{1, \ldots, n\}$ are independent \citep[Chapter~1, Problem~8]{CheneyLight00}, i.e., the matrix $\h(\x)$ is invertible.  

\begin{proposition}
\label{P:T2S}
Let $f$ be a real-valued function defined in a neighbourhood of infinity and let $\g$ be a rate vector. If $T = (0, \infty)$ and if the functions $h_1, \ldots, h_n$ in \eqref{E:f} are linearly independent, then $S = (0, \infty)$, and for $y \in (0, \infty)$ and for any $\x \in (0, \infty)^n$ such that $\h(\x)$ is invertible,
\begin{align}
\label{E:f2g}
	\g(t) &= \h(\x)^{-1} \{ f(t \x) - f(t \1) \} + o(g_n(t)), \\
\label{E:h2A}
	\A(y) &= \h(\x)^{-1} \{ \h(y \x) - \h(y \1) \}.
\end{align}
\end{proposition}

\begin{proof}
For $\x \in (0, \infty)^n$ and $i \in \{1, \ldots, n\}$, we have
\[
	f(x_i t) = f(t) + \h(x_i) \g(t) + o(g_n(t));
\]
in matrix notation, this becomes
\begin{equation}
\label{E:f:vector}
	f(t \x) = f(t \1) + \h(\x) \g(t) + o(g_n(t)).
\end{equation}
Now let $\x \in (0, \infty)^n$ be such that $\h(\x)$ is invertible. Then \eqref{E:f:vector} clearly implies \eqref{E:f2g}. For $y \in (0, \infty)$, we have by \eqref{E:f2g} and \eqref{E:f:vector}
\begin{align}
\label{E:f2g:aux}
	\g(ty) 
	&= \h(\x)^{-1} \{ f(ty \x) - f(ty \1) \} + o(g_n(yt)) \nonumber \\
	&= \h(\x)^{-1} \{ f(ty \x) - f(t \1) \} - \h(\x)^{-1} \{ f(ty \1) - f(t \1) \} + o(g_n(yt)) \nonumber \\
	&= \h(\x)^{-1} \h(y \x) \g(t) - \h(\x)^{-1} \h(y \1) \g(t) + o(g_n(t)) + o(g_n(yt)) \nonumber \\
	&= \A(y) \g(t) + o(g_n(t)) + o(g_n(yt)),
\end{align}
with $\A(y) = \h(\x)^{-1} \{ \h(y \x) - \h(y \1) \}$. Row number $i$ of \eqref{E:f2g:aux} reads
\[
	g_i(ty) = \sum_{j=1}^n A_{ij}(y) g_j(t) + o(g_n(t)) + o(g_n(ty)).
\]
In particular, $g_1(ty) = A_{11}(y) g_1(t) + o(g_1(t)) + o(g_1(ty))$ and thus
\[
	g_1(ty) / g_1(t) \to A_{11}(y).
\]	
For $i \in \{2, \ldots, n\}$, we find from the two previous displays that
\[
	A_{i1}(y) = \lim_{t \to \infty} g_i(ty) / g_1(t) = 0
\]
and thus
\[
	g_i(ty) = \sum_{j=2}^n A_{ij}(y) g_j(t) + o(g_n(t)) + o(g_n(ty)), \qquad i \in \{2, \ldots, n\}
\]
Repeating the same argument inductively yields eventually
\[
	g_n(ty) = A_{nn}(y) g_n(y) + o(g_n(t)) + o(g_n(ty)),
\]
from which
\[
	g_n(ty) / g_n(t) \to A_{nn}(y).
\]
Hence, on the right-hand side in \eqref{E:f2g:aux}, the term $o(g_n(ty))$ can be absorbed by the term $o(g_n(t))$, yielding $\g(ty) = \A(y) \g(t) + o(g_n(t))$. Hence $y \in S$ with $\A(y)$ as in \eqref{E:h2A}.
\end{proof}

\begin{remark}
\label{R:linearindependence}
Let $f$ be a real-valued function defined in a neighbourhood of infinity, let $\g$ be a rate vector, and assume that not all functions $h_1,\ldots,h_n$ in \eqref{E:f} are identically zero. Then by basic linear algebra, there exists a non-empty subset $I$ of $\{1, \ldots, n\}$ such that the following holds: 
\begin{enumerate}
\item[(a)] the functions $h_i$, $i \in I$ are linearly independent;
\item[(b)] there exist $\lambda_{ij} \in \RR$, $i \in \{1, \ldots, n\}$ and $j \in I \cap \{1, \ldots, i\}$, such that $h_i = \sum_{i \in I, j \leq i} \lambda_{ij} h_j$ (the empty sum being zero zero by convention). 
\end{enumerate}
Then
\[
	\h(x) \g(t) 
	= \sum_{i=1}^n \sum_{j \in I, j \leq i} \lambda_{ij} h_j(x) g_i(t)
	= \sum_{j \in I} h_j(x) \tilde{g}_j(t)
\]
with $\tilde{g}_j(t) = \sum_{i=j}^n \lambda_{ij} g_i(t)$. Since $\lambda_{jj} = 1$ for $j \in I$, we have $\tilde{g}_j(t) / g_j(t) \to 1$; in particular, $(\tilde{g}_j)_{j \in I}$ is a rate vector.
\end{remark}

\begin{remark}
\label{R:Tx}
Let $f$ be a real-valued function defined in a neighbourhood of infinity, let $\g$ be a rate vector, and assume that there exists $\x \in T^n$ such that $\h(\x)$ is invertible. Let 
\begin{equation}
\label{E:Tx}
	T(\x) = \{ y \in T : x_i y \in T, \forall i \}. 
\end{equation}
By repeating the proof of Proposition~\ref{P:T2S}, we see that $T(\x) \subset S$. As a consequence, if $T(\x)$ contains a set of positive measure and if both $T(\x) \cap (0, 1)$ and $T(\x) \cap (1, \infty)$ are non-empty, then by Proposition~\ref{P:S}(c) and Proposition~\ref{P:S2T} actually $S = (0, \infty)$ and $T = (0, \infty)$.
\end{remark}

\section{Regularly varying rate vectors and generalised regular variation}
\label{S:measurability}

Up to now, we have studied the asymptotic relations \eqref{E:g} and \eqref{E:f} without assuming that $f$ or $\g$ are measurable. In this section, we will add this assumption and also assume that $S = (0, \infty)$ and $T = (0, \infty)$. We pay particular attention to the functional forms of the limit functions $\A$ and $\h$. These forms admit simple expressions in terms of matrix exponential functions.

Some preliminaries on matrix exponentials are collected in Subsection~\ref{SS:matrixexp}. In Subsection~\ref{SS:rvrv}, we focus on the rate vector $\g$. Finally, in Subsection~\ref{SS:grv}, we turn to the function $f$ itself.

\subsection{Matrix exponentials}
\label{SS:matrixexp}

For a square matrix $\B$, recall the matrix exponential functions
\begin{align*}
	\exp(\B) &= \sum_{k = 0}^\infty \frac{1}{k!} \B^k, \\
	x^\B &= \exp \{ (\log x) \B \}, \qquad x \in (0, \infty).
\end{align*}
By convention, $\B^0 = \I$, the identity matrix. If $\B$ is upper triangular, the entries of $x^\B$ can be computed in a recursive way.

\begin{proposition}
\label{P:B2A}
Let $\B \in \RR^{n \times n}$ be upper triangular and put $\A(x) = x^\B$, $x \in (0, \infty)$. Then $A_{ii}(x) = x^{b_i}$ with $b_i = B_{ii}$, $i \in \{1, \ldots, n\}$, while for $i \in \{1, \ldots, n-1\}$,
\begin{equation}
\label{E:B2A}
	(A_{i,i+1}(x), \ldots, A_{i,n}(x))
	= x^{b_i} (B_{i,i+1}, \ldots, B_{i,n}) \int_1^x \A_{i+1,n}(y) y^{-b_i-1} \dy,
\end{equation}
where $\A_{kl}(x) = (A_{ij}(x))_{i,j = k}^l$. In particular, if all diagonal elements of $\B$ are zero, then for integer $i$ and $k$ such that $1 \leq i < i+k \leq n$ and for $x \in (0, \infty)$,
\begin{equation}
\label{E:xB:Pi}
	A_{i, i+k}(x) = \sum_{l = 1}^k \frac{(\log x)^l}{l!} \sum_{i = j_0 < \cdots < j_l = i+k} \prod_{m=1}^l B_{j_{m-1}, j_m},
\end{equation}
the inner sum being over all $(l+1)$-tuples of positive integers $(j_0, \ldots, j_l)$ satisfying the stated (in)equalities.
\end{proposition}

\begin{proof}
The starting point is the relation $\A(xy) = \A(x) \A(y)$, from which it follows that $\A(xy) - \A(x) = \{\A(y) - \I\} \A(x)$. We obtain
\[
	\frac{\A(xy) - \A(x)}{y-1} = \frac{\A(y) - I_n}{y - 1} \A(x), \qquad x, y \in (0, \infty).
\]
Taking limits as $y \to 1$, we find that $x \dot{\A}(x) = \B \A(x)$, and therefore, since $\B$ and $\A(x)$ are upper triangular,
\begin{equation}
\label{E:diffeq}
	\dot{A}_{ij}(x) = x^{-1} \sum_{k = i}^{j} B_{ik} A_{kj}(x), \qquad x \in (0, \infty), \quad 1 \leq i \leq j \leq n.
\end{equation}
This system of differential equations can be solved as follows. If $i = j$, equation~\eqref{E:diffeq} becomes
\[
	\dot{A}_{ii}(x) = x^{-1} b_i A_{ii}(x), \qquad x \in (0, \infty).
\]
In combination with the initial condition $A_{ii}(1) = 1$, this implies $A_{ii}(x) = x^{b_i}$.

Next assume $1 \leq i < j \leq n$. Rewrite \eqref{E:diffeq} as
\[
	\dot{A}_{ij}(x) - x^{-1} b_i A_{ij}(x) = x^{-1} \sum_{k=i+1}^{j} B_{ik} A_{kj}(x), \qquad x \in (0, \infty).
\]
Looking for solutions of the form $A_{ij}(x) = C_{ij}(x) x^{b_i}$, we find that $C_{ij}(x)$ should satisfy 
\[
	\dot{C}_{ij}(x) = x^{-b_i-1} \sum_{k = i+1}^{j} B_{ik} A_{kj}(x), \qquad x \in (0, \infty).
\]
Since $C_{ij}(1)=0$, we obtain
\[
	C_{ij}(x) = \sum_{k = i+1}^{j} B_{ik} \int_{1}^{x} A_{kj}(y) y^{-b_i-1} \dy, \qquad x \in (0, \infty)
\]
and consequently
\[
	A_{ij}(x) = x^{b_i} \sum_{k = i+1}^{j} B_{ik} \int_{1}^{x} A_{kj}(y) y^{-b_i-1} \dy, \qquad x \in (0, \infty),
\]
which is \eqref{E:B2A}.

Next suppose that all diagonal elements of $\B$ are zero. Then all diagonal elements of $\A(x)$ are equal to unity, and by \eqref{E:B2A},
\begin{align*}
	A_{i,i+k}(x) 
	&= \sum_{j=i+1}^{i+k} B_{ij} \int_1^x A_{j,i+k}(u) u^{-1} \du \\
	&= \sum_{j=i+1}^{i+k-1} B_{ij} \int_1^x A_{j,i+k}(u) u^{-1} \du + B_{i,i+k} \log x.
\end{align*}
We proceed by induction on $k$. If $k = 1$, then the above display tells us that $A_{i,i+1}(x) = B_{i,i+1} \log x$, which is \eqref{E:xB:Pi}. If $k \geq 2$, then use of the induction hypothesis and the previous display again leads, after some algebra, to the desired equality.
\end{proof}

\begin{remark}
\label{R:B2h}
\begin{flushenumerate}
\item[(a)]
Let $\B \in \RR^{n \times n}$ and write $\bm{H}(x) = \int_1^x u^\B u^{-1} \du$. Term-by-term integration of the series expansion $u^\B = \sum_{k=0}^\infty (\log u)^k B^k / k!$ yields the convenient formula
\begin{equation}
\label{E:Bh}
	\B \bm{H}(x) = \H(x) \B = x^\B - \I, \qquad x \in (0, \infty).
\end{equation}
If $\B$ is invertible, then
\begin{equation}
\label{E:B2h:inv}
	\bm{H}(x) = \B^{-1} (x^\B - \I) = (x^\B - \I) \B^{-1}.
\end{equation}
\item[(b)]
Suppose that $\bm{D} \in \RR^{(n+1) \times (n+1)}$ is given by
\[
	\bm{D} =
	\begin{pmatrix}
		0 & \c \\
		\0 & \B
	\end{pmatrix}
\]
where $\c \in \RR^{1 \times n}$ and where $\B \in \RR^{n \times n}$ is upper triangular. Then $\bm{D}$ is upper triangular as well, and by Proposition~\ref{P:B2A}, for $x \in (0, \infty)$,
\[
	x^{\bm{D}} =
	\begin{pmatrix}
		1 & \h(x) \\
		\0 & x^\B
	\end{pmatrix},
\]
where $\h(x) \in \RR^{1 \times n}$ is given by
\[
	\h(x) = \c \int_1^x y^\B y^{-1} \dy.
\]
\end{flushenumerate}
\end{remark}

\subsection{Regularly varying rate vectors}
\label{SS:rvrv}

\begin{proposition}
\label{P:B:a}
Let $\g$ be a measurable rate vector. If there exists $\A : (0, \infty) \to \RR^{n \times n}$ such that
\[
	\g(xt) = \A(x) \g(t) + o(g_n(t)), \qquad x \in (0, \infty),
\]
then there exists an upper triangular matrix $\B \in \RR^{n \times n}$ such that
\[
	\A(x) = x^{\B}, \qquad x \in (0, \infty).
\]
\end{proposition}

\begin{proof}
From \eqref{E:g2A:a}--\eqref{E:g2A:b} it follows that $\A$ is measurable. Hence, in view of Proposition~\ref{P:A}, $\A$ is a measurable group homomorphism from $(0, \infty)$ into the group of invertible $n \times n$ matrices. Then necessarily $\A(x) = x^\B$, $x \in (0, \infty)$, with
\[
	\B = \lim_{x \to 1} \frac{\A(x) - \I}{x - 1},
\]
see for instance \citet[Theorem~VIII.1.2 and Lemma~VIII.1.3]{DunfordSchwartz58}. From the above display, it follows that $\B$ is upper triangular.
\end{proof}

\begin{proposition}
\label{P:B:b}
Let $\B$ be the matrix appearing in Proposition~\ref{P:B:a} and put $b_i = B_{ii}$, $i \in \{1, \ldots, n\}$.
\begin{flushenumerate}
\item[(a)] The function $|g_i|$ is regularly varying with index $b_i$; in particular, $b_1 \geq \cdots \geq b_n$.
\item[(b)] All the eigenvalues of $\B$, except maybe for the smallest one, have geometric multiplicity equal to one.
\end{flushenumerate}
\end{proposition}

\begin{proof}
(a) This follows from the fact that
\[
	g_i(xt) / g_i(t) \to A_{ii}(x) = x^{b_i}, 
	\qquad x \in (0, \infty), \quad i \in \{1, \ldots, n\}.
\]
Since $g_i(t) = o(g_{i-1}(t))$, necessarily $b_i \leq b_{i-1}$.

(b) Let $b$ be an eigenvalue (diagonal element) of $\B$ and let $\v \in \RR^{1 \times n}$ satisfy $\v\B = b\v$. Then $\v x^\B = x^b \v$ for all $x > 0$. Hence
\[
	\v \g(xt) = \v x^\B \g(t) + o(g_n(t)) = x^b \v\g(t) + o(g_n(t)), \qquad x \in (0, \infty),
\]
and thus
\[
	(xt)^{-b} \v\g(xt) = t^{-b} \v\g(t) + o(t^{-b} g_n(t)), \qquad x \in (0, \infty).
\]
This states that the function $t \mapsto t^{-b} \v\g(t)$ belongs to the class $o\Pi_a$ with $a(t) = t^{-b} g_n(t)$. Now unless $b$
is equal to the smallest eigenvalue, $b_n$, the Representation Theorem for $o\Pi$ \citep[Theorem~3.6.1]{BGT} stipulates the existence of $c \in \RR$ such that
\[
	t^{-b} \v\g(t) = c + o(t^{-b}g_n(t)),
\]
or in other words
\[
	\v\g(t) = ct^b + o(g_n(t)).
\]
Now let $b > b_n$ and suppose that $\v_1, \v_2 \in \RR^{1 \times n}$ both satisfy $\v_i \B = b \v_i$, for $i = 1,2$. Let $c_1, c_2 \in \RR$ be such that $\v_i \g(t) = c_i t^b + o(g_n(t))$, for $i = 1,2$. Find real numbers $\lambda_1$ and $\lambda_2$, not both zero, such that $\lambda_1 c_1 + \lambda_2 c_2 = 0$. The vector $\v = \lambda_1 \v_1 + \lambda_2 \v_2$ clearly satisfies $\v \g(t) = o(g_n(t))$. As $\g$ is a rate vector, necessarily $\v = \bm{0}$. Hence $\v_1$ and $\v_2$ are linearly dependent.
\end{proof}

These results motivate the following definition.

\begin{definition}
\label{D:RV}
A matrix $\B \in \RR^{n \times n}$ is an \emph{index matrix} if it is upper triangular, if its diagonal elements are non-increasing, and if all of its eigenvalues, except maybe for the smallest one, have geometric multiplicity equal to one.

A rate vector $\g$ of length $n$ is \emph{regularly varying with index matrix $\B$} if $\g$ is measurable and 
\begin{equation}
\label{E:rv}
	\g(xt) = x^\B \g(t) + o(g_n(t)), \qquad x \in (0, \infty). 
\end{equation}
Notation: $\g \in \RV_\B$.
\end{definition}

Combining Propositions~\ref{P:S}(c), \ref{P:B:a}, and \ref{P:B:b}, we arrive at our first main result.

\begin{theorem}[(Characterisation Theorem for $\RV_\B$)]
\label{T:rv:char}
Let $\g$ be a measurable rate vector of length $n$ and let $S$ be the set of $x \in (0, \infty)$ for which there exists $\A(x) \in \RR^{n \times n}$ such that $\g(xt) = \A(x) \g(t) + o(g_n(t))$. If $S$ contains a set of positive measure and if both $S \cap (0,1)$ and $S \cap (1,\infty)$ are non-empty, then $\g \in \RV_\B$ for some index matrix $\B$.
\end{theorem}

\begin{remark}
\label{R:rv:char}
Let $\g \in \RV_\B$.
\begin{flushenumerate}
\item[(a)] If $\bm{Q} \in \RR^{n \times n}$ is upper triangular and invertible, then $\bm{Q}\g$ is a rate vector as well and $\bm{Q}\g \in \RV_{\bm{Q} \B \bm{Q}^{-1}}$.
\item[(b)] For integer $1 \leq k \leq l \leq n$, the subvector $\g_{kl} = (g_k, \ldots, g_l)'$ is a rate vector as well and $\g_{kl} \in \RV_{\B_{kl}}$, where $\B_{kl} = (B_{ij})_{i,j = k}^l$.
\end{flushenumerate}
\end{remark}

\subsection{Generalised regular variation}
\label{SS:grv}

Now we take up again the study of the relation $f(xt) = f(t) + \h(x) \g(t) + o(g_n(t))$, this time for measurable $f$. In view of Proposition~\ref{P:T2S}, there is not much harm in assuming from the start that $\g$ is measurable as well.

\begin{definition}
A measurable, real-valued function $f$ defined in a neighbourhood of infinity is \emph{generalised regularly varying with rate vector $\g \in \RV_\B$ and $\g$-index $\c \in \RR^{1 \times n}$} if
\begin{equation}
\label{E:grv}
	f(xt) = f(t) + \h(x) \g(t) + o(g_n(t)), \qquad x \in (0, \infty),
\end{equation}
where
\begin{equation}
\label{E:B2h}
	\h(x) = \c \int_1^x y^\B y^{-1} \dy, \qquad x \in (0, \infty).
\end{equation}
Notation: $f \in \GRV(\g)$.
\end{definition}

Our second main result asserts that the regular variation property of $\g$ and the form of the limit function $\h$ in \eqref{E:B2h} are in some sense built in in the relation \eqref{E:grv}. The assumption in Theorem~\ref{T:grv:char} that the limit functions $h_1, \ldots, h_n$ are linearly independent is unavoidable; however, in view of Remark~\ref{R:linearindependence}, unless all $h_i$ are identically zero, it is always possible to switch to a subvector $(h_i)_{i \in I}$ of functions that are linearly independent.

\begin{theorem}[(Characterisation Theorem for $\GRV(\g)$)]
\label{T:grv:char}
Let $f$ be a measurable, real-valued function defined in a neighbourhood of infinity and let $\g$ be a measurable rate vector of length $n$. Let $T$ be the set of $x \in (0, \infty)$ for which there exists $\h(x) \in \RR^{1 \times n}$ such that $f(tx) = f(t) + \h(x) \g(t) + o(g_n(t))$. Assume that there exists $\x \in T^n$ such that
\begin{enumerate}
\item[(i)] the matrix $\h(\x)$ in \eqref{E:hx} is invertible;
\item[(ii)] the set $T(\x)$ in \eqref{E:Tx} contains a set of positive measure and has non-empty intersections with both $(0, 1)$ and $(1, \infty)$.
\end{enumerate}
Then $\g \in \RV_\B$ and $f \in \GRV(\g)$ with $\h$ as in \eqref{E:B2h}. In addition, all eigenvalues (including the smallest one) of the index matrix $\B$ have geometric multiplicity equal to one.
\end{theorem}

\begin{proof}
By Remark~\ref{R:Tx}, necessarily $T = (0, \infty)$ and there exists $\A : (0, \infty) \to \RR^{n \times n}$ such that $\g(tx) = \A(x) \g(t) + o(g_n(t))$ for all $x \in (0, \infty)$. By Theorem~\ref{T:rv:char}, $\g \in \RV_\B$ for some index matrix $\B \in \RR^{n \times n}$. 

By Remark~\ref{R:C}, the matrix function $\bm{C} : (0, \infty) \to \RR^{(n+1) \times (n+1)}$ in \eqref{E:C} is a measurable group homomorphism. As in the proof of Proposition~\ref{P:B:a}, we find that $\bm{C}(x) = x^{\bm{D}}$ for some upper triangular matrix $\bm{D} \in \RR^{(n+1) \times (n+1)}$. Since $C_{11}(x) = 1$ and $(C_{ij}(x))_{i,j=2}^{n+1} = \A(x) = x^\B$, the expression for $h$ in \eqref{E:B2h} follows from Remark~\ref{R:B2h}.

Let $b$ be an eigenvalue (i.e.\ a diagonal element) of $\B$ and let $\v \in \RR^{n \times 1}$ be such that $\B \v = b \v$. For $y \in (0, \infty)$, we have $y^\B \v = y^b \v$. As a consequence, for $x \in (0, \infty)$,
\[
	\h(x) \v 
	= \c \int_1^x y^\B y^{-1} \dy \, \v
	= \int_1^x y^{b-1} \dy \, \c \v.
\]
Now let both $\v_1, \v_2 \in \RR^{n \times 1}$ be eigenvectors of $\B$ with the same eigenvalue $b$. There exist $\lambda_1, \lambda_2 \in \RR$, not both zero, such that $\lambda_1 \c \v_1 + \lambda_2 \c \v_2 = 0$, and thus
\[
	\h(x) (\lambda_1 \v_1 + \lambda_2 \v_2) = \int_1^x y^{b-1} \dy \, (\lambda_1 \c \v_1 + \lambda_2 \c \v_2) = 0, 
	\qquad x \in (0, \infty).
\]
Since the functions $h_1, \ldots, h_n$ were assumed to be linearly independent, the above identity implies that $\lambda_1 \v_1 + \lambda_2 \v_2 = \bm{0}$, that is, $\v_1$ and $\v_2$ are linearly dependent. As a consequence, the dimension of the eigenspace of $b$ cannot be larger than one.
\end{proof}

\begin{remark}
\label{R:grv:char}
Let $f \in \GRV(\g)$ with $\g$-index $\c \in \RR^{1 \times n}$. If $\Q \in \RR^{n \times n}$ is upper triangular and invertible, then also $f \in \GRV(\Q \g)$ with $\Q \g$-index $\c \Q^{-1}$.
\end{remark}

\section{Uniformity, representations, and Potter bounds}
\label{S:URP}
\subsection{Uniform convergence theorems}
\label{SS:UCT}

If $\g \in \RV_\B$, then the asymptotic relation \eqref{E:rv} holds locally uniformly in $x \in (0, \infty)$. Similarly, if $f \in \GRV(\g)$, then the asymptotic relation \eqref{E:grv} holds locally uniformly in $x \in (0, \infty)$. The proof of Theorem~\ref{T:rv:UCT} is inspired by the one in \citet{Delange55} for the uniform convergence theorem for slowly varying functions as presented in \citet[Theorem~1.2.1]{BGT}.

\begin{theorem}[(Uniform convergence theorem for $\RV_\B$)]
\label{T:rv:UCT}
If $\g\in \RV_\B$, then $\g(tx) = x^{B} \g(t) + o(g_{n}(t))$ locally uniformly in $x \in (0, \infty)$.
\end{theorem}

\begin{proof}
We will prove uniform convergence for $x \in [a^{-1},a] $ where $a \in (1, \infty)$. For $\v=(v_{1},v_{2},...,v_{n})$ let $\| 
\v \| = \max ( |v_1|, \ldots, |v_n| )$. The corresponding operator norm on $\RR ^{n\times n}$ will also be denoted by $\| \, \cdot \, \|$. Put $b = B_{nn}$ and choose $\eps \in (0, \log a)$. For $t > 0$ sufficiently large, we define the sets
\begin{align*}
	E(t) &= \{ s \in [ a^{-1}t, at] : \| \g(s)-(s/t)^{B}\g(t)\| \geq \eps |g_{n}(t)| \}, \\
	V(t) &= \{ x \in [ a^{-1}, a] :\| \g(tx) - x^{B}\g(t) \| \geq \eps |g_{n}(t)| \}.
\end{align*}%
These sets are measurable and $E(t)= \{ tx : x \in V(t) \} =t V(t)$. With $\mu$ the measure on $(0, \infty)$ defined by $\mu(\dy) = y^{-1}\dy$, we have $\mu (E(t)) = \mu(V(t))$. As $\g \in \RV_\B$, it follows that the indicator function of $V(t)$ converges pointwise to zero. By dominated convergence, we have $\mu(V(t)) \to 0$ and we can find $t_{0} > 0$ such that $\mu(E(t)) = \mu (V(t)) \leq \eps /2$ for all $t \geq t_{0}$.

For $a^{-1}\leq x\leq a$, the intersection $[ a^{-1}xt, axt ] \cap [a^{-1}t, at]$ contains at least one of the intervals $[a^{-1}t, t] $ or $[t, at]$. This implies that $\mu ([a^{-1}xt,axt] \cap [ a^{-1}t, at]) \geq \log a$. On the other hand, if $t \geq at_{0} = t_{1}$ and $x \geq a^{-1}$, then $\mu (E(tx)\cup E(t)) \leq \eps$. Now as $\eps < \log a$, for all $x \in [a^{-1}, a]$ and $t \geq t_{1}$, the set
\[
	V(x,t) = ([ a^{-1}xt, axt ] \cap [a^{-1}t, at]) \setminus (E(xt)\cup E(t))
\]
has positive $\mu$-measure and so is certainly non-empty. Let $s=s(x,t) \in V(x,t)$. By definition of $E(t)$ we have for $x\in [a^{-1}, a] $ and $t \geq t_{1}$,
\begin{align*}
	\| \g(s)-(s/t)^{\B}\g(t)\| &< \eps | g_{n}(t) |, \\
	\| \g(s)-(s/(xt))^{\B}\g(xt) \| &< \eps | g_{n}(xt)|,
\end{align*}
whence, by the triangle inequality, 
\[
	\| (s/(xt))^\B \{ \x^\B \g(t) - \g(xt) \} \| < 2 \eps \{ |g_n(t)| + |g_n(xt)| \}.
\]
Since the function $| g_{n} |$ is regularly varying with index $b$, by the Uniform convergence theorem for regularly varying functions \citep[Theorem~1.2.1]{BGT}, there exists $t_2 \geq t_1$ such that
\[
	| g_{n}(tx)| \leq 2 a^{b} |g_{n}(t)|, \qquad  x \in [a^{-1}, a], \quad t \geq t_{2}.
\]
Combine the last two displays to see that
\[
	\| (s/(xt))^\B \{ \x^\B \g(t) - \g(xt) \} \| < 2 \eps (1 + a^b) |g_n(t)|,
	\qquad x \in [a^{-1}, a], \quad t \geq t_2.
\]
Applying the inequality $\| \bm{T} \v\| \geq \| \bm{T}^{-1}\| ^{-1}\| \v\| $, valid for invertible $\bm{T} \in \RR^{n\times n}$, we get from the previous display that
\[
	\| x^{B} \g(t)-\g(tx)\| \leq 2 \eps (1 + a^b) |g_n(t)| \, \| (xt/s)^{\B} \|, 
	\qquad x \in [a^{-1}, a], \quad t \geq t_{2}.
\]
Now note that for such $x$ and $t$ we have $a^{-2} \leq xt/s \leq a^{2}$. Since the function $0 < y \mapsto y^{\B}$ is continuous, there is a positive constant $C = C(a,\B)$ such that 
\[
	\| x^{B}\g(t)-\g(tx)\| \leq \eps C |g_n(t)|,
	\qquad x \in [a^{-1}, a], \quad t \geq t_{2}.
\]
This proves the result.
\end{proof}

\begin{theorem}[(Uniform convergence theorem for $\GRV(\g)$)]
\label{T:grv:UCT}
If $f \in \GRV(\g)$, then $f(tx) = f(t) + \h(x) \g(t) + o(g_n(t))$ holds locally uniformly in $x \in (0, \infty)$.
\end{theorem}

\begin{proof}
In view of Remark~\ref{R:C}, we can recycle the proof of Theorem~\ref{T:rv:UCT}. Note that in that proof we nowhere used the fact that $\g$ is a rate vector, but only that $\g(xt) = \x^\B \g(t) + o(g_n(t))$ for all $x \in (0, \infty)$ and some square matrix $\B$ as well as regular variation of $|g_n|$.
\end{proof}

\subsection{Representation theorems}
\label{SS:repr}

\begin{theorem}[(Representation theorem for $\RV_\B$)]
\label{T:rv:repr}
Let $\g$ be a rate vector of length $n$ and let $\B \in \RR^{n \times n}$ be an index matrix. Then $\g \in \RV_\B$ if and only there exist $a \in (0, \infty)$, $\v \in \RR^{n \times 1}$ and measurable functions $\etab, \phib : [a, \infty) \to \RR^{n \times 1}$, both $o(g_n(t))$, such that
\begin{equation}
\label{E:rv:repr}
	\g(t) = t^\B \v + \etab(t) + t^\B \int_a^t u^{-\B} \phib(u) u^{-1} \du, \qquad t \in [a, \infty).
\end{equation}
\end{theorem}

\begin{proof}
\emph{Necessity.} Assume that $\g \in \RV_\B$. Since the functions $|g_i|$ are regularly varying, they are locally bounded on $[a, \infty)$ for some sufficiently large $a \in (0, \infty)$. Define
\begin{align*}
	\v &= \int_a^{ea} u^{-\B} \g(u) u^{-1} \du, \\
	\etab(t) &= \int_1^e \{ \g(t) - u^{-B} \g(ut) \} u^{-1} \du, \\
	\phib(t) &= e^{-B} \g(et) - \g(t).
\end{align*}
By Theorem~\ref{T:rv:UCT}, the vector valued functions $\etab(t)$ and $\phib(t)$ are both $o(g_n(t))$. The equality \eqref{E:rv:repr} can be verified by simple algebra.

\paragraph{Sufficiency.} 
Suppose that the rate vector $\g$ admits the representation \eqref{E:rv:repr} with $\etab(t)$ and $\phib(t)$ both $o(g_n(t))$. Put $b = B_{nn}$. Since $\B$ is upper triangular, row number $n$ of \eqref{E:rv:repr} reads
\[
	g_n(t) = t^b v_n + \eta_n(t) + t^b \int_a^t u^{-b} \phi_n(u) u^{-1} \du, \qquad t \in [a, \infty).
\]
For $u \in (1, \infty)$ and $t \in [a, \infty)$, we have
\begin{align}
	g_n(ut) - u^b g_n(t) 
	&=  \eta_n(ut) - u^b \eta_n(t)
	+ (ut)^b \int_t^{ut} y^{-b-1} \phi_n(y) \dy \nonumber \\
\label{E:rv:repr:10}
	&=  \eta_n(ut) - u^b \eta_n(t)
	+ u^b \int_1^u y^{-b-1} \phi_n(yt) \dy.
\end{align}
Fix $x \in (1, \infty)$. We have
\begin{align*}
	\sup_{u \in [1, x]} |\eta_n(ut)|
	&= \sup_{u \in [1, x]} \frac{|\eta_n(ut)|}{|g_n(ut)|} |g_n(ut)| \\
	&\leq \sup_{u \in [1, x]} \frac{| \eta_n(ut)|}{|g_n(ut)|} 
	\cdot \sup_{u \in [1, x]} |g_n(ut)|
\end{align*}
and thus
\[
	\sup_{u \in [1, x]} |\eta_n(ut)|
	= o \biggl( \sup_{u \in [1, x]} |g_n(ut)| \biggr).
\]
Similarly for $\eta_n$ replaced by $\phi_n$. In view of \eqref{E:rv:repr:10}, it follows that
\[
	\sup_{u \in [1, x]} |g_n(tu)|
	\leq x^b |g_n(t)| + o \biggl( \sup_{u \in [1, x]} |g_n(ut)| \biggr),
\]
and thus
\begin{equation}
\label{E:rv:repr:20}
	\sup_{u \in [1, x]} |g_n(tu)| = O(g_n(t)).
\end{equation}
Now let us look at the complete rate vector $\g$. As in \eqref{E:rv:repr:10}, we find, again for $x \in (1, \infty)$,
\begin{equation}
\label{E:rv:repr:30}
	\g_n(xt) - x^\B \g_n(t) 
	=  \etab(xt) - x^\B \etab(t)
	+ x^\B \int_1^x u^{-\B} \phib(ut) u^{-1} \du.
\end{equation}
From \eqref{E:rv:repr:20} and the assumption that both $\etab(t)$ and $\phib(t)$ are $o(g_n(t))$, the above display implies
\[
	\g_n(xt) - x^\B \g_n(t) = o(g_n(t)), \qquad x \in [1, \infty).
\]
Since $x^\B$ is invertible, the above display and Proposition~\ref{P:A} imply that $\g \in \RV_\B$.
\end{proof}

\begin{remark}
The case $n = 1$ in Theorem~\ref{T:rv:repr} seems to be a new representation for regularly varying functions. The representation is the same as the one for the class $o\Pi_g$ in \citet[Theorem~3.6.1]{BGT}, but with the difference that the function $g$ is not assumed to be of bounded increase. Indeed, the main point in the proof of sufficiency was precisely to show that the representation actually implies that $g$ is of bounded increase, see \eqref{E:rv:repr:20}.
\end{remark}

\begin{theorem}[(Representation theorem for $\GRV(\g)$)]
\label{T:grv:repr}
Let $\g \in \RV_\B$ and let $f$ be a measurable, real-valued function defined in a neighbourhood of infinity. Then $f \in \GRV(\g)$ with $\g$-index $\c \in \RR^{1 \times n}$ if and only if there exist constants $a \in (0, \infty)$ and $v \in \RR$ as well as measurable functions $\eta, \phi : [a, \infty) \to \RR$, both $o(g_n(t))$, such that
\[
	f(t) = v + \eta(t) + \int_a^t \{ \c \g(u) + \phi(u) \} u^{-1} \du, \qquad t \in [a, \infty).
\]
\end{theorem}

\begin{proof}
Let $a \in (0, \infty)$ be large enough so that the domain of $\g$ includes $[a, \infty)$ and put
\begin{align*}
	\tilde{f}(t) &= \c \int_a^t \g(u) u^{-1} \du, \\
	\xi(t) &= f(t) - \tilde{f}(t),
\end{align*}
for $t \in [a, \infty)$. For $x \in (0, \infty)$ and $t$ large enough so that $t \geq a$ and $xt \geq a$,
\[
	\tilde{f}(xt) - \tilde{f}(t)
	= \c \int_t^{xt} \g(u) u^{-1} \du
	= \c \int_1^x \g(ut) u^{-1} \du.
\]
By the Uniform convergence theorem for $\RV_\B$,
\[
	\tilde{f}(xt) - \tilde{f}(t)
	= \c \int_1^x u^\B u^{-1} \du \g(t) + o(g_n(t)), \qquad x \in (0, \infty),
\]
and thus $\tilde{f} \in \GRV(\g)$ with $\g$-index $\c$. It follows that $f \in \GRV(\g)$ with $\g$-index $\c$ if and only if
\[
	\xi(xt) - \xi(t) = o(g_n(t)), \qquad x \in (0, \infty),
\]
that is, $\xi \in o\Pi_{|g_n|}$. The Representation theorem for $o\Pi_{|g_n|}$ \citep[Theorem~3.6.1]{BGT} says that the above display is equivalent to the existence of a constant $v \in \RR$ and measurable functions $\eta, \phi : [a, \infty) \to \RR$, both $o(g_n(t))$, such that
\[
	\xi(t) = v + \eta(t) + \int_a^t \phi(u) u^{-1} \du.
\]
Since $f = \tilde{f} + \xi$, we arrive at the desired representation.
\end{proof}

\subsection{Potter bounds}
\label{SS:Potter}

The representation theorems for $\RV_\B$ and $\GRV(\g)$ allow us to derive global upper bounds for $\norm{ \g(xt) - x^\B \g(t) } / |g_n(t)|$ and $|f(xt) - f(t) - \h(x) \g(t)| / |g_n(t)|$. In analogy to classical regular variation theory, such kind of bounds will be called Potter bounds.

First recall that for any matrix $\Q \in \RR^{n \times n}$ and any matrix norm $\norm{ \, \cdot \, }$,
\[
	\lim_{m \to \infty} \norm{ \Q^m }^{1/m} = \max \{ |\lambda| : \text{$\lambda$ is an eigenvalue of $\Q$} \}.
\]
Now let $\B \in \RR^{n \times n}$ be an upper triangular matrix whose diagonal elements $b_i = B_{ii}$, $i \in \{1, \ldots, n\}$, are non-increasing, $b_1 \geq \cdots \geq b_n$. For $x \in (0, \infty)$, the eigenvalues of $x^\B$ and $x^{-\B}$ are $\{x^{b_i}\}$ and $\{x^{-b_i}\}$, respectively. The above display then implies
\begin{align}
\label{E:xB:1}
	\lim_{x \to \infty} \frac{\log \norm{x^\B}}{\log x} &= b_1, \\
\label{E:xB:n}
	\lim_{x \to 0} \frac{\log \norm{x^\B}}{\log x} &= \lim_{y \to \infty} \frac{\log \norm{y^{-\B}}}{- \log y} = b_n.
\end{align}

\begin{theorem}[(Potter bounds for $\RV_\B$)]
\label{T:rv:Potter}
Let $\g \in \RV_\B$. For every $\eps > 0$, there exists $t(\eps) > 0$ such that
\begin{equation}
\label{E:rv:Potter}
	\frac{\norm{\g(xt) - x^\B \g(t)}}{|g_n(t)|} \leq
	\begin{cases}
		\eps x^{b_1 + \eps} 
			& \text{if $t \geq t(\eps)$ and $x \geq 1$;} \\
		\eps x^{b_n - \eps} 
			& \text{if $t \geq t(\eps)$ and $t(\eps)/t \leq x \leq 1$.}
	\end{cases}
\end{equation}
\end{theorem}

\begin{proof}
We shall prove the following two statements, which are equivalent to \eqref{E:rv:Potter}: for $\delta > 0$,
\begin{align}
\label{E:rv:Potter:a}
	\lim_{t \to \infty} \sup_{x \geq 1} \frac{\norm{\g(xt) - x^\B \g(t)}}{x^{b_1 + \delta} |g_n(t)|} &= 0, \\
\label{E:rv:Potter:b}
	\lim_{t \to \infty} \sup_{x \geq 1} \frac{\norm{\g(t) - x^{-\B} \g(xt)}}{x^{- b_n + \delta} |g_n(xt)|} &= 0.
\end{align}
Let $a$, $\etab$ and $\phib$ be as in the Representation theorem for $\RV_\B$ (Theorem~\ref{T:rv:repr}).

\paragraph{Proof of \eqref{E:rv:Potter:a}.} 
For all $t \geq a$ and $x \geq 1$,
\begin{align*}
	\g(xt) - x^\B \g(t)
	&= \etab(xt) - x^\B \etab(t) + (xt)^\B \int_t^{xt} u^{-\B} \phib(u) u^{-1} \du \\
	&= \etab(xt) - x^\B \etab(t) + x^\B \int_1^x u^{-\B} \phib(ut) u^{-1} \du
\end{align*}
and thus, since $b_n \leq b_1$,
\begin{multline*}
	\frac{\norm{\g(xt) - x^\B \g(t)}}{x^{b_1 + \delta} |g_n(t)|} 
	\leq \frac{\norm{\etab(xt)}}{|g_n(xt)|} \frac{|g_n(xt)|}{x^{b_n + \delta}|g_n(t)|} 
		+ \frac{\norm{x^\B}}{x^{b_1 + \delta}} \frac{\norm{\etab(t)}}{|g_n(t)|} \\
		+ \frac{\norm{x^\B}}{x^{b_1 + \delta/2}} 
		\int_1^x 
			\frac{\norm{u^{-\B}}}{u^{-b_n} x^{\delta/4}} 
			\frac{\norm{\phib(ut)}}{|g_n(ut)|} 
			\frac{|g_n(ut)|}{u^{b_n} x^{\delta/4} |g_n(t)|} 
		u^{-1} \du.
\end{multline*}
As a consequence, for $t \geq a$,
\begin{multline*}
	\sup_{x \geq 1} \frac{\norm{\g(xt) - x^\B \g(t)}}{x^{b_1 + \delta}|g_n(t)|} \\
	\shoveleft{\qquad 
	\leq \sup_{s \geq t} \frac{\norm{\etab(s)}}{|g_n(s)|} \cdot
	\sup_{x \geq 1} \frac{|g_n(xt)|}{x^{b_n + \delta}|g_n(t)|}
	+ \sup_{x \geq 1} \frac{\norm{x^\B}}{x^{b_1 + \delta}} \cdot \frac{\norm{\etab(t)}}{|b_n(t)|}
	} \\
	+ \sup_{x \geq 1} \frac{\norm{x^\B}}{x^{b_1 + \delta/4}} \cdot
		\sup_{u \geq 1} \frac{\norm{u^{-\B}}}{u^{-b_n + \delta/4}} \cdot
		\sup_{s \geq t} \frac{\norm{\phib(s)}}{|g_n(s)|} \cdot
		\sup_{u \geq 1} \frac{|g_n(ut)|}{u^{b_n + \delta/4} |g_n(t)|} \cdot
		\sup_{x \geq 1} \frac{\log x}{x^{\delta/4}}.
\end{multline*}
By Potter's theorem for $\RV_{b_n}$ and by \eqref{E:xB:1}--\eqref{E:xB:n}, for every $\eps > 0$,
\[
	\lim_{t \to \infty} \sup_{x \geq 1} \frac{|g_n(xt)|}{x^{b_n + \eps} |g_n(t)|} = 1, \qquad
	\sup_{x \geq 1} \frac{\norm{x^\B}}{x^{b_1 + \eps}} < \infty, \qquad
	\sup_{x \geq 1} \frac{\norm{x^{-\B}}}{x^{-b_n + \eps}} < \infty.
\]
Combine the last two displays and the fact that both $\norm{\etab(t)}$ and $\norm{\phib(t)}$ are $o(g_n(t))$ to arrive at \eqref{E:rv:Potter:a}.

\paragraph{Proof of \eqref{E:rv:Potter:b}.}
By \eqref{E:rv:repr}, for all $t \geq a$ and $x \geq 1$,
\begin{align*}
	\g(t) - x^{-\B} \g(xt)
	&= \etab(t) - x^{-\B} \etab(xt) + t^\B \int_t^{xt} u^{-\B} \phib(u) u^{-1} \du \\
	&= \etab(t) - x^{-\B} \etab(xt) + \int_1^x u^{-\B} \phib(ut) u^{-1} \du
\end{align*}
and thus
\begin{multline*}
	\frac{\norm{\g(t) - x^{-\B} \g(xt)}}{x^{- b_n + \delta}|g_n(xt)|}
	\leq \frac{\norm{\etab(t)}}{|g_n(t)|} \frac{|g_n(t)|}{x^{-b_n + \delta} |g_n(xt)|}
	+ \frac{\norm{x^{-\B}}}{x^{-b_n + \delta}} \frac{\norm{\etab(xt)}}{|g_n(xt)|} \\
	+ \frac{|g_n(t)|}{x^{-b_n+\delta/2}|g_n(xt)|}
		\int_1^x 
			\frac{\norm{u^{-\B}}}{u^{-b_n} x^{\delta/4}} 
			\frac{\norm{\phib(ut)}}{|g_n(ut)|}
			\frac{|g_n(ut)|}{u^{b_n} x^{\delta/4}|g_n(t)|}
		u^{-1} \du.
\end{multline*}
As a consequence, for $t \geq a$,
\begin{multline*}
	\sup_{x \geq 1} \frac{\norm{\g(t) - x^{-\B} \g(xt)}}{x^{- b_n + \delta}|g_n(xt)|} \\
	\shoveleft{\qquad
	\leq \frac{\norm{\etab(t)}}{|g_n(t)|} \cdot \sup_{x \geq 1} \frac{|g_n(t)|}{x^{-b_n + \delta} |g_n(xt)|}
	+ \sup_{x \geq 1} \frac{\norm{x^{-\B}}}{x^{-b_n + \delta}} \cdot \sup_{s \geq t} \frac{\norm{\etab(s)}}{|g_n(s)|}
	} \\
	+ \sup_{x \geq 1} \frac{|g_n(t)|}{x^{-b_n+\delta/4}|g_n(xt)|} \cdot
		\sup_{u \geq 1} \frac{\norm{u^{-\B}}}{u^{-b_n + \delta/4}} \cdot
		\sup_{s \geq t} \frac{\norm{\phib(s)}}{|g_n(s)|} \cdot
		\sup_{u \geq 1} \frac{|g_n(ut)|}{u^{b_n + \delta/4}|g_n(t)|} \cdot
		\sup_{x \geq 1} \frac{\log x}{x^{\delta/4}}.
\end{multline*}
By Potter's theorem for $\RV_{b_n}$ and by \eqref{E:xB:n}, for every $\eps > 0$,
\[
	\lim_{t \to \infty} \sup_{x \geq 1} \frac{|g_n(t)|}{x^{-b_n + \eps} |g_n(xt)|} = 1, \qquad
	\lim_{t \to \infty} \sup_{x \geq 1} \frac{|g_n(xt)|}{x^{b_n + \eps} |g_n(t)|} = 1, \qquad
	\sup_{u \geq 1} \frac{\norm{u^{-\B}}}{u^{-b_n + \eps}} < \infty.
\]
Combine the last two displays and the fact that both $\norm{\etab(t)}$ and $\norm{\phib(t)}$ are $o(g_n(t))$ to arrive at \eqref{E:rv:Potter:b}.
\end{proof}

\begin{remark}
\label{R:rv:Potter}
\begin{flushenumerate}
\item[(a)]
Since
\[
	\frac{\norm{\g(t) - x^{-\B} \g(xt)}}{x^\delta |g_n(t)|}
	= \frac{|g_n(xt)|}{x^{b_n + \delta/2} |g_n(t)|} \cdot \frac{\norm{\g(t) - x^{-\B} \g(xt)}}{x^{-b_n + \delta/2} |g_n(xt)|},
\]
equation~\eqref{E:rv:Potter:b} and Potter's theorem for $\RV_{b_n}$ imply
\begin{equation}
\label{E:rv:Potter:c}
	\lim_{t \to \infty} \sup_{x \geq 1} \frac{\norm{\g(t) - x^{-\B} \g(xt)}}{x^\delta |g_n(t)|} = 0,
	\qquad \delta > 0.
\end{equation}
\item[(b)]
Equation~\eqref{E:rv:Potter:a} can be used to give a simple proof of a statement which is slightly weaker than \eqref{E:rv:Potter:b}. For $t \geq a$ and $x \geq 1$,
\[
	\frac{\norm{\g(t) - x^{-\B} \g(xt)}}{x^{b_1 - 2b_n + \delta}|g_n(xt)|}
	\leq \frac{\norm{x^{-\B}}}{x^{-b_n + \delta/3}} \cdot
	\frac{|g_n(t)|}{x^{-b_n + \delta/3} |g_n(xt)|} \cdot
	\frac{\norm{\g(xt) - x^\B \g(t)}}{x^{b_1 + \delta/3} |g_n(t)|}.
\]
By \eqref{E:xB:n} and by Potter's theorem for $\RV_{b_n}$, for every $\eps > 0$,
\[
	\sup_{x \geq 1} \frac{\norm{x^{-\B}}}{x^{-b_n + \eps}} < \infty, \qquad
	\lim_{t \to \infty} \sup_{x \geq 1} \frac{|g_n(t)|}{x^{-b_n + \eps} |g_n(xt)|} = 1.
\]
Combine the last two displays and \eqref{E:rv:Potter:a} to arrive at
\[
	\lim_{t \to \infty} \sup_{x \geq 1} \frac{\norm{\g(t) - x^{-\B} \g(xt)}}{x^{b_1 - 2b_n + \delta}|g_n(xt)|} = 0,
	\qquad \delta > 0.
\]
\end{flushenumerate}
\end{remark}

For $a, b \in \RR$, write $a \vee b = \max(a, b)$.

\begin{theorem}[(Potter bounds for $\GRV(\g)$)]
\label{T:grv:Potter}
Let $\g \in \RV_\B$ and let $f \in \GRV(\g)$. For every $\eps > 0$, there exists $t(\eps) > 0$ such that
\begin{equation}
\label{E:grv:Potter}
	\frac{|f(xt) - f(t) - \h(x) \g(t)|}{|g_n(t)|} \leq
	\begin{cases}
		\eps x^{(b_1 + \eps) \vee 0} 
			& \text{if $t \geq t(\eps)$ and $x \geq 1$;} \\
		\eps x^{(b_n - \eps) \wedge 0}
			& \text{if $t \geq t(\eps)$ and $t(\eps)/t \leq x \leq 1$.}
	\end{cases}
\end{equation}
\end{theorem}

\begin{proof}
We shall prove the following two statements, which are equivalent to \eqref{E:grv:Potter}: for $\delta > 0$,
\begin{align}
\label{E:grv:Potter:a}
	\lim_{t \to \infty} \sup_{x \geq 1}
	\frac{|f(xt) - f(t) - \h(x) \g(t)|}{x^{(b_1 + \delta) \vee 0} |g_n(t)|} &= 0, \\
\label{E:grv:Potter:b}
	\lim_{t \to \infty} \sup_{x \geq 1}
	\frac{|f(t) - f(xt) - \h(x^{-1}) \g(xt)|}{x^{(-b_n + \delta) \vee 0} |g_n(xt)|} &= 0.
\end{align}
Recall that $\h(x) = \c \int_1^x u^\B u^{-1} \du$ for $x \in (0, \infty)$, with $\c \in \RR^{1 \times n}$ the $\g$-index of $f$. Also, recall the representation of $f$ in Theorem~\ref{T:grv:repr}. 

\paragraph{Proof of \eqref{E:grv:Potter:a}.}
For $t \geq a$ and $x \geq 1$,
\begin{equation}
\label{E:grv:Potter:10}
	f(xt) - f(t) - \h(x) \g(t)
	= \c \int_1^x \{ \g(ut) - u^\B \g(t) \} u^{-1} \du + \eta(xt) - \eta(t) + \int_1^x \phi(ut) u^{-1} \du
\end{equation}
and thus, since $b_n + \delta \leq b_1 + \delta \leq (b_1 + \delta) \vee 0$,
\begin{multline*}
	\frac{|f(xt) - f(t) - \h(x) \g(t)|}{x^{(b_1 + \delta) \vee 0} |g_n(t)|} \\
	\leq 
			\frac{\norm{\c}}{x^{(b_1 + \delta) \vee 0}}
			\int_1^x \frac{\norm{\g(ut) - u^\B \g(t)}}{u^{b_1 + \delta/2} |g_n(t)|} u^{b_1+\delta/2-1} \du
		+ \frac{|\eta(xt)|}{|g_n(xt)|} \cdot \frac{|g_n(xt)|}{x^{b_n + \delta}|g_n(t)|}
		+ \frac{|\eta(t)|}{|g_n(t)|} \\
		+ \frac{1}{x^{(b_1 + \delta) \vee 0}}
			\int_1^x \frac{|\phi(ut)|}{|g_n(ut)|} \frac{|g_n(ut)|}{u^{b_n + \delta/2} |g_n(t)|} u^{b_n + \delta/2 - 1} \du.
\end{multline*}
We obtain that, for $t \geq a$,
\begin{multline*}
	\sup_{x \geq 1} \frac{|f(xt) - f(t) - \h(x) \g(t)|}{x^{(b_1 + \delta) \vee 0} |g_n(t)|} \\
	\leq 
			\sup_{u \geq 1} \frac{\norm{\g(ut) - u^\B \g(t)}}{u^{b_1 + \delta/2} |g_n(t)|} \cdot
			\sup_{x \geq 1} \frac{\norm{\c}}{x^{(b_1 + \delta) \vee 0}} \int_1^x u^{b_1 + \delta/2 - 1} \du \\
		+ \sup_{s \geq 1} \frac{|\eta(s)|}{|g_n(s)|} \cdot 
			\sup_{x \geq 1} \frac{|g_n(xt)|}{x^{b_n + \delta}|g_n(t)|}
		+ \frac{|\eta(t)|}{|g_n(t)|} \\
		+ \sup_{s \geq 1} \frac{|\phi(s)|}{|g_n(s)|} \cdot 
			\sup_{u \geq 1} \frac{|g_n(ut)|}{u^{b_n + \delta/2} |g_n(t)|} \cdot
			\sup_{x \geq 1} \frac{1}{x^{(b_1 + \delta) \vee 0}} \int_1^x u^{b_n + \delta/2 - 1} \du.
\end{multline*}
In view of Potter's theorem for $\RV_{b_n}$ and for $\RV_\B$ (Theorem~\ref{T:rv:Potter}), it suffices to note that $\int_1^x u^{b_1 + \delta/2 - 1} \du = O(x^{(b_1 + \delta) \vee 0})$ as $x \to \infty$.

\paragraph{Proof of \eqref{E:grv:Potter:b}.} By \eqref{E:hx1}, $\h(x^{-1}) = - \h(x) x^{-\B}$ and thus $\h(x^{-1}) \g(xt) = - \h(x) x^{-\B} \g(xt) = - \c \int_1^x u^\B u^{-1} \du \, x^{-\B} \g(xt) = - \c \int_1^x (x/u)^{-\B} \g(xt) u^{-1} \du$ for $x \geq 1$ and $t \geq a$. We obtain that
\begin{multline*}
	f(t) - f(xt) - \h(x^{-1}) \g(xt) \\
	= \eta(t) - \eta(xt) - \int_1^x \phi(ut) u^{-1} \du
	- \c \int_1^x \{ \g(ut) - (x/u)^{-\B} \g(xt) \} u^{-1} \du,
	\qquad t \geq a, \quad x \geq 1.
\end{multline*}
We obtain that, for $t \geq a$ and $x \geq 1$,
\begin{align*}
	\lefteqn{
	\frac{|f(t) - f(xt) - \h(x^{-1}) \g(xt)|}{x^{(-b_n + \delta) \vee 0} |g_n(xt)|}  
	} \\
	&\leq 
		\frac{|\eta(t)|}{|g_n(t)|} \frac{|g_n(t)|}{x^{-b_n + \delta} |g_n(xt)|}
	+ \frac{|\eta(xt)|}{|g_n(xt)|} \\
	&\qquad \mbox{} + \frac{1}{x^{\{(-b_n + \delta) \vee 0 \} + b_n - \delta/2}}
		\int_1^x \frac{|\phi(ut)|}{|g_n(ut)|} \frac{|g_n(ut)|}{(x/u)^{-b_n + \delta/2} |g_n(xt)|} u^{b_n - \delta/2 - 1} \du \\
	&\qquad \mbox{} + \frac{\norm{\c}}{x^{\{ (-b_n + \delta) \vee 0 \} + b_n - \delta/2}} 
		\int_1^x 
			\frac{\norm{\g(ut) - (x/u)^{-\B} \g(xt)}}{(x/u)^{\delta/4} |g_n(ut)|} 
			\frac{|g_n(ut)|}{(x/u)^{-b_n + \delta/4}|g_n(xt)|}
		u^{b_n - \delta/2 - 1} \du.
\end{align*}
Since $\{ (-b_n + \delta) \vee 0 \} + b_n - \delta/2 = (b_n - \delta/2) \vee (\delta/2)$, we find, for $t \geq a$,
\begin{align*}
	\lefteqn{
	\sup_{x \geq 1} \frac{|f(t) - f(xt) - \h(x^{-1}) \g(xt)|}{x^{(-b_n + \delta) \vee 0} |g_n(xt)|} 
	} \\
	&\leq 		
		\frac{|\eta(t)|}{|g_n(t)|} \cdot \sup_{x \geq 1} \frac{|g_n(t)|}{x^{-b_n + \delta} |g_n(xt)|}
	+ \sup_{s \geq 1} \frac{|\eta(s)|}{|g_n(s)|} \\
	&\qquad \mbox{} + \sup_{s \geq 1} \frac{|\phi(s)|}{|g_n(s)|} \cdot
		\sup_{s \geq t} \sup_{y \geq 1} \frac{|g_n(s)|}{y^{-b_n + \delta/2} |g_n(ys)|} \cdot
		\frac{1}{x^{(b_n - \delta/2) \vee (\delta/2)}} \int_1^x u^{b_n - \delta/2 - 1} \du \\
	&\qquad \mbox{} + \sup_{s \geq 1} \sup_{y \geq 1} \frac{\norm{\g(s) - y^{-\B} \g(ys)}}{y^{\delta/4} |g_n(s)|} \cdot
	\sup_{s \geq 1} \sup_{y \geq 1} \frac{|g_n(s)}{y^{-b_n + \delta/4} |g_n(ys)|} \cdot
		\frac{1}{x^{(b_n - \delta/2) \vee (\delta/2)}} \int_1^x u^{b_n - \delta/2 - 1} \du.	
\end{align*}
Now apply the following elements:
\begin{enumerate}
\item[(a)] the assumption that both $\eta(t)$ and $\phi(t)$ are $o(g_n(t))$ (Theorem~\ref{T:grv:repr});
\item[(b)] Potter's theorem for $\RV_{b_n}$;
\item[(c)] Potter's theorem for $\RV_\B$ [equation~\eqref{E:rv:Potter:c} in Remark~\ref{R:rv:Potter}];
\item[(d)] the fact that for $a \in \RR$ and $\eps > 0$, $\int_1^x u^{a - 1} \du = O(x^{a \vee \eps})$ as $x \to \infty$;
\end{enumerate}
to arrive at \eqref{E:grv:Potter:b}.
\end{proof}

\section{$\Pi$-variation of arbitrary order. I. General functions}
\label{S:PiI}

An interesting special case of (generalised) regular variation of arbitrary order arises when all diagonal elements of the index matrix $\B$ are equal to zero. In this case, the Potter bounds in \eqref{E:rv:Potter} and \eqref{E:grv:Potter} simplify in an obvious way. The case $n = 2$ has been studied in \cite{OmeyWillekens88}.

\begin{notation}
If $\g \in \RV_\B$ and if all diagonal elements of $\B$ are equal to $0$, then, in analogy to the first-order case, we write $\Pi(\g)$ rather than $\GRV(\g)$.
\end{notation}

If all diagonal elements of $\B$ are equal to $b$, nonzero, then $\g \in \RV_\B$ if and only if the rate vector $t \mapsto t^{-b} \g(t)$ belongs to $\RV_{\B - b\I}$, the index matrix $\B - b\I$ having zeroes on the diagonal. For $f \in \Pi(\g)$, however, such a reduction does not occur; instead, see Remark~\ref{R:speccas:nonzero}.

For general $n \geq 2$, by Proposition~\ref{P:B2A} and Remark~\ref{R:B2h}, if $\g \in \RV_\B$ and if $\B$ has a zero diagonal, then for each $i \in \{1, \ldots, n-1\}$, the function $g_i$ belongs to the class $\Pi((g_{i+1}, \ldots, g_n)')$ with index $(B_{i,i+1}, \ldots, B_{i,n})$.

\subsection{Reduction to Jordan block index matrices}
\label{SS:Pi:Jordan}

In case of the $n \times n$ Jordan block
\begin{equation}
\label{E:Jordan}
	\J_{n} =
	\begin{pmatrix}
	0      & 1      & 0      & \ldots & \ldots & 0 \\ 
	0      & 0      & 1      & 0      & \ldots & 0 \\ 
	\vdots &        & \ddots & \ddots & \ddots & \vdots \\ 
	0      & \ldots & \ldots & 0      & 1      & 0 \\
	0      & \ldots & \ldots & \ldots & 0      & 1 \\ 
	0      & \ldots & \ldots & \ldots & \ldots & 0
	\end{pmatrix},
\end{equation}
we find by direct computation
\begin{equation}
\label{E:xJ}
	x^{\J_n} = \I + \sum_{k=1}^{n-1} \frac{(\log x)^k}{k!} \J_n^k =
	\begin{pmatrix}
	1 & \log x & \frac{1}{2!} (\log x)^{2} & \ldots & \frac{1}{(n-1)!}(\log x)^{n-1} \\ 
	0 & 1 & \log x & \ldots & \frac{1}{(n-2)!}(\log x)^{n-2} \\ 
	\ldots & \ldots & \ldots & \ldots & \ldots \\ 
	0 & \ldots & 0 & 1 & \log x \\ 
	0 & \ldots & 0 & 0 & 1
	\end{pmatrix}.
\end{equation}
Hence, $\g \in \RV_{\J_n}$ if and only if for every $i \in \{1, \ldots, n\}$,
\[
	g_i(xt) = g_i(t) + \sum_{k = 1}^{n-i} \frac{(\log x)^k}{k!} g_{i+k}(t) + o(g_n(t)), \qquad x \in (0, \infty).
\]

For general upper triangular $\B$ with zero diagonal, $(x^\B)_{i,i+k}$ is a polynomial in $\log x$ of degree at most $k$, the coefficient of $(\log x)^k / k!$ being $B_{i,i+1} B_{i+1, i+2} \cdots B_{i+k-1, i+k}$; see \eqref{E:xB:Pi}. The highest power of $\log x$ which can arise in $x^\B$ is $n-1$, the coefficient of $(\log x)^{n-1} / (n-1)!$ being
\[
	B_{1,2} B_{2,3} \cdots B_{n-1, n},
\]
the product of all elements on the superdiagonal. This product is nonzero if and only if $\B$ has no zeroes on the superdiagonal. According to the next proposition, this is equivalent to the condition that the eigenvalue $0$ of $\B$ has geometric multiplicity equal to one, a property which is guaranteed in the setting of the Characterisation Theorem for $\GRV(\g)$, Theorem~\ref{T:grv:char}. In this case, regular variation with index matrix $\B$ can be reduced to regular variation with index matrix $\J_n$.

\begin{proposition}
\label{P:Jordan}
Let $\B \in \RR^{n \times n}$ be upper triangular and with zeroes on the diagonal. The following are equivalent:
\begin{flushenumerate}
\item[(i)] $B_{i,i+1} \neq 0$ for every $i \in \{1, \ldots, n-1\}$.
\item[(ii)] There exists an invertible, upper triangular matrix $\Q \in \RR^{n \times n}$ such that $\B = \Q \J_n \Q^{-1}$.
\item[(iii)] The eigenvalue $0$ of $\B$ has geometric multiplicity equal to one.
\item[(iv)] If $\A(x) = x^\B$, then the functions $A_{1j}$, $j \in \{2, \ldots, n\}$, are linearly independent.
\end{flushenumerate}
In this case, $\g \in \RV_\B$ if and only if $\Q^{-1} \g \in \RV_{\J_n}$.
\end{proposition}

\begin{proof}
\textit{(i) implies (ii).} The proof is by induction on $n$. For $n=1$ the statement is trivial: just take any nonzero number $Q = Q_{11}$. 

Let $n \geq 2$ and write $\tilde{\B} = (B_{ij})_{i,j=1}^{n-1}$. The induction hypothesis implies that we can find and invertible, upper triangular matrix $\tilde{\Q} \in \RR^{(n-1) \times (n-1)}$ such that $\tilde{\B} \tilde{\Q} = \tilde{\Q} \J_{n-1}$. Let $\bm{b} = (B_{in})_{i=1}^{n-1} \in \RR^{(n-1) \times 1}$. For $\bm{q} \in \RR^{(n-1) \times 1}$ and $Q_{nn}$ to be determined, write
\begin{align*}
	\B &= \begin{pmatrix} \tilde{\B} & \bm{b} \\ \bm{0} & 0 \end{pmatrix}, & 
	\Q &= \begin{pmatrix} \tilde{\Q} & \bm{q} \\ \bm{0} & Q_{nn} \end{pmatrix}.
\end{align*}
We have to show that we can choose $\bm{q}$ and $Q_{nn}$ in such a way that $Q_{nn} \neq 0$ and
\[
	\B \Q = \Q \J_n.
\]
By the decomposition of $\B$ and $\Q$ above and by the property of $\tilde{\Q}$, the equation in the preceding display reduces to
\begin{equation}
\label{E:Jordan:n-1}
	\tilde{\B} \bm{q} + Q_{nn} \bm{b} = (Q_{i,n-1})_{i=1}^{n-1}.
\end{equation}
Since the last row of $\tilde{\B}$ contains zeroes only, the $(n-1)$th equation of the previous display is just $Q_{nn} B_{n-1,n} = Q_{n-1,n-1}$. Since $B_{n-1,n} \neq 0$ (by assumption) and since $Q_{n-1,n-1} \neq 0$ (by the induction hypothesis), we find $Q_{nn} = Q_{n-1,n-1} / B_{n-1,n}$. Let $\hat{\B}$ be the $(n-2) \times (n-2)$ submatrix of $\B$ with range $i = 1, \ldots, n-2$ and $j = 2, \ldots, n-1$. The first $n-2$ equations in \eqref{E:Jordan:n-1} are
\[
	\hat{\B} (Q_{in})_{i=2}^{n-1} = (Q_{i,n-1})_{i=1}^{n-2} - Q_{nn} (B_{in})_{i=1}^{n-2}.
\]
Since $\hat{\B}$ is an upper triangular matrix with nonzeroes on the diagonal, the above system has a unique solution in $(Q_{in})_{i=2}^{n-1}$. The element $Q_{1n}$ can be chosen arbitrarily.

\textit{(ii) implies (iii).} The vector $\bm{v} \in \RR^{n \times 1}$ is an eigenvector of $\B$ with eigenvalue $0$ if and only if the vector $\Q^{-1} \bm{v}$ is an eigenvector of $\J_n$ with the same eigenvalue. The dimension of the eigenspace of $\J_n$ corresponding to eigenvalue $0$ is equal to one.

\textit{(iii) implies (i).} The vector $(1, 0, \ldots, 0)' \in \RR^{n \times 1}$ is an eigenvector of $\B$ with eigenvalue $0$. Now suppose that there exists $i \in \{1, \ldots, n-1\}$ such that $B_{i,i+1} = 0$. We will construct an eigenvector of $\B$ with eigenvalue $0$ which is linearly independent of $(1, 0, \ldots, 0)'$.

Consider the $i \times i$ submatrix $\tilde{\B} = (B_{kl} : 1 \leq k \leq i, 2 \leq l \leq i+1)$. Since $B_{i,i+1} = 0$, the last row of $\tilde{\B}$ contains only zeroes. Hence there exists a non-zero vector $\tilde{\bm{v}} \in \RR^{i \times 1}$ such that $\tilde{\B} \tilde{\bm{v}} = \bm{0}$. Put $\bm{v} = (0, \tilde{\bm{v}}', 0, \ldots, 0)' \in \RR^{n \times 1}$. The $\B \bm{v} = \bm{0}$ and $\bm{v}$ is linearly independent of $(1, 0, \ldots, 0)'$.

\textit{(i) and (iv) are equivalent.} By \eqref{E:xB:Pi}, $A_{1j}(x)$ is a polynomial in $\log x$ of degree at most $j-1$, without constant term, the coefficient of $(\log x)^j / j!$ being $\prod_{i=1}^{j-1} B_{i,i+1}$. Hence, the component functions $A_{1j}$, $j \in \{2, \ldots, n\}$, are linearly independent if and only if $\prod_{i=1}^{j-1} B_{i,i+1} \neq 0$ for all $j \in \{2, \ldots, n\}$. But this is equivalent to $\prod_{i=1}^{n-1} B_{i,i+1} \neq 0$, which is (i).
\end{proof}

If $\g \in \RV_\B$ with $\B$ having a zero diagonal, then for $i \in \{1, \ldots, n-1\}$, as $(x^\B)_{i,i+1} = B_{i,i+1} \log x$,
\[
	g_i(xt) = g_i(t) + B_{i,i+1} \log (x) g_{i+1}(t) + o(g_{i+1}(t)), \qquad x \in (0, \infty).
\]
Hence, the condition on $\B$ in Proposition~\ref{P:Jordan} means that for every $i \in \{1, \ldots, n-1\}$, the function $g_i$ is up to its sign in the (univariate) de Haan class $\Pi$ with auxiliary slowly varying function $|g_{i+1}|$.

\subsection{Representation via indices transform}
\label{SS:Pi:indices}

As in de Haan's Theorem \citep[Theorem~3.7.3]{BGT} and in \citet[Theorem~2.1]{OmeyWillekens88}, we consider the relation between the asymptotic behaviour of a function $f$, locally integrable on $[a, \infty)$ for some $a > 0$, and the function $C$ defined by
\begin{equation}
\label{E:f2C}
	C(t) = f(t) - \frac{1}{t} \int_{a}^{t}f(y)\dy, \qquad t \in [a, \infty).
\end{equation}
The map from $f$ to $C$ is a special case of the indices transform \citep[equation~(3.5.2)]{BGT}. Equation~\eqref{E:f2C} can be inverted as follows: 
\begin{equation}
\label{E:C2f}
	f(x) = C(x) + \int_{a}^{x}C(t)t^{-1}\dt, \qquad x \in [a, \infty).
\end{equation}
According to the next theorem, if $f$ is $\Pi$-varying of order $n$, then \eqref{E:C2f} can be seen as a representation of $f$ in terms of a function $C$ which is $\Pi$-varying of order at most $n-1$. Note that the condition that $f$ is (eventually) locally integrable entails no loss of generality, for local integrability is a consequence of the representation for $f$ in Theorem~\ref{T:grv:repr}.

\begin{theorem}[(Representation theorem for $\Pi(\g)$)]
\label{T:Pi:repr}
Let $\B$ be an index matrix with zero diagonal, let $\g \in \RV_\B$, and let $f$ be locally integrable on $[a, \infty)$ for some $a > 0$. Let $C$ be as in \eqref{E:f2C}. Then $f \in \Pi(\g)$ with $\g$-index $\c$ if and only if
\begin{equation}
\label{E:g2C}
	C(t) = \bm{d} \g(t) + o(g_n(t))
\end{equation} 
with $\bm{d} = \c (\I + \B)^{-1}$. In particular, $C \in \Pi((g_2, \ldots, g_n)')$ with index $((\bm{d}\B)_i)_{i=2}^n$.
\end{theorem}

\begin{proof}
Since the entries of $x^\B$ are polynomials in $\log v$, see \eqref{E:xB:Pi}, the integral $\bm{K} = \int_0^1 x^\B \dx$ is well-defined. By \eqref{E:Bh}, we find
\begin{align*}
	\K = \int_0^1 (x^\B - \I) \dx + \I 
	&= - \B \int_0^1 \int_x^1 u^\B u^{-1} \du \dx + \I \\
	&= - \B \int_0^1 \int_0^u \dx u^\B u^{-1} \du + \I
	= - \B \K + \I,
\end{align*}
whence
\begin{equation}
\label{E:K}
	\int_0^1 x^\B \dx = (\I + \B)^{-1}.
\end{equation}

\textsl{Sufficiency.}
Suppose that $C$ satisfies \eqref{E:g2C}. Recall $\bm{H}(x) = \int_1^x u^\B u^{-1} \du$ in Remark~\ref{R:B2h}(a). By the uniform convergence theorem for $\RV_\B$ (Theorem~\ref{T:rv:UCT}), we have $C(ut) = \bm{d} u^\B \g(t) + o(g_n(t))$ locally uniformly in $u \in (0, \infty)$. As a consequence, for $x \in (0, \infty)$ and $t$ sufficiently large,
\begin{align*}
	f(xt) - f(t)
	&= C(xt) - C(t) + \int_t^{xt} C(u) u^{-1} \du \\
	&= \bm{d} x^\B \g(t) - \bm{d} \g(t) + \int_1^x \bm{d} u^\B \g(t) u^{-1} \du + o(g_n(t)) \\
	&= \bm{d} \{ x^\B - \I + \bm{H}(x) \} \g(t) + o(g_n(t)).
\end{align*}
By \eqref{E:Bh},
\[
	f(xt) - f(t) 
	= \bm{d} (\I + \B) \bm{H}(x) \g(t) + o(g_n(t))
	= \c \bm{H}(x) \g(t) + o(g_n(t)).
\]

\textsl{Necessity.} 
Suppose that $f \in \Pi(\g)$ with $\g$-index $\c$. Write $\h(x) = \c \int_1^x u^\B u^{-1} \du$. By \eqref{E:K}, we have
\begin{align*}
	\int_0^1 \h(x) \dx 
	&= - \c \int_0^1 \int_x^1 u^\B u^{-1} \du \, \dx
	= - \c \int_0^1 \int_0^u \dx \, u^\B u^{-1} \du \\
	&= - \c \int_0^1 u^\B \du
	= - \c (\I + \B) = - \bm{d}.
\end{align*}
Fix $\eps > 0$. By Potter's theorem for $\Pi(\B) = \GRV(\B)$, there exists $t_\eps \geq a$ such that
\[
	|f(xt) - f(t) - \h(x) \g(t)| \leq \eps x^{-\eps} |g_n(t)|,
	\qquad t \geq t_\eps, \quad t_\eps / t \leq x \leq 1.
\]
For $t \geq t_\eps$, we have
\begin{align}
\label{E:C:decomp}
	C(t) - \bm{d} \g(t)
	&= f(t) - \frac{1}{t} \int_a^t f(y) \dy + \int_0^1 \h(x) \dx \g(t) \nonumber \\
	&= t_\eps \frac{f(t)}{t} - \frac{1}{t} \int_a^{t_\eps} f(y) \dy \nonumber \\
	& \qquad \mbox{} - \int_{t_\eps/t}^1 \{f(xt) - f(t) - \h(x) \g(t) \} \dx + \int_0^{t_\eps/t} \h(x) \dx \g(t).
\end{align}
Recall that all the functions $|g_i|$ are slowly varying. The first term in the display above is $O(t^{-1})$ and thus $o(g_n(t))$. Since $f \in (o)\Pi_{|g_1|}$, the function $|f|$ cannot grow faster than a slowly varying function, whence also $f(t) / t = o(g_n(t))$. By the Potter bound above, the third term is bounded by $\eps \int_0^1 x^{-\eps} \dx |g_n(t)| = \{\eps/(1-\eps)\} |g_n(t)|$. Finally, since the entries of $u^\B$ are polynomials in $\log u$, see \eqref{E:xB:Pi}, we have $\int_0^s \h(x) \dx = o(s^{1-\eps})$ as $s \downarrow 0$; since all the functions $|g_i|$ are slowly varying, the last term in \eqref{E:C:decomp} must be $o(g_n(t))$ as well.

\textsl{$\Pi$-variation of $C$.}
Using \eqref{E:Bh} and \eqref{E:g2C}, we find for $x \in (0, \infty)$,
\[
	C(xt) - C(t) 
	= \bm{d} (x^\B - \I) \g(t) + o(g_n(t))
	= \bm{d} \B \int_1^x u^\B u^{-1} \du \, \g(t) + o(g_n(t)).
\]
Since $(x^\B)_{i1} = \delta_{i1}$ for $i \in \{1, \ldots, n\}$, the function $g_1$ disappears from the above equation. It follows that $C \in \Pi((g_2, \ldots, g_n)')$, that is, $C$ is in the class $\Pi$ but of order at most $n-1$; the $(g_2, \ldots, g_n)'$-index of $C$ is given by $((\bm{d} \B)_i)_{i = 2}^n$.
\end{proof}

Now let us iterate the procedure in Theorem~\ref{T:Pi:repr}. Define functions $C_0, C_1, \ldots, C_n$ recursively as follows:
\begin{align}
\label{E:Ci}
	C_0 &= f; &
	C_i(t) &= C_{i-1}(t) - \frac{1}{t} \int_a^t C_{i-1}(y) \dy, 
	\qquad t \in [a, \infty), \, i \in \{1, \ldots, n\}.
\end{align}
In words, $C_i$ is the result of applying the indices transform in \eqref{E:f2C} to $C_{i-1}$. According to Theorem~\ref{T:Pi:repr:n} below, for $\Pi$-varying functions $f$ of order $n$, the `canonical' situation is the following: the vector $\C = (C_1, \ldots, C_n)'$ is itself a regularly varying rate vector with index matrix $\K_n \in \RR^{n \times n}$ given by
\begin{equation}
\label{E:Kn}
	\K_n =
	\begin{pmatrix}
		0      & 1      &      1 & \ldots & 1      \\
		0      & 0      &      1 & \ldots & 1      \\
		\vdots &        & \ddots & \ddots & \vdots \\
		0      & \ldots & \ldots & 0      & 1 \\
		0      & \ldots & \ldots & \ldots & 0
	\end{pmatrix};
\end{equation}
moreover, $f \in \Pi(\C)$ with $\C$-index $(1, \ldots, 1) \in \RR^{1 \times n}$. The condition in Theorem~\ref{T:Pi:repr:n} that $B_{i,i+1} \neq 0$ for all $i \in \{1, \ldots, n-1\}$ was found to be a natural one in Proposition~\ref{P:Jordan}: according to Theorem~\ref{T:grv:char}, it is guaranteed if the limit functions $h_1, \ldots, h_n$ are linearly independent.

\begin{theorem}[(Iterating the indices transform)]
\label{T:Pi:repr:n}
Let $\B \in \RR^{n \times n}$ be an index matrix with zero diagonal, let $\g \in \RV_\B$, and let $f$ be locally integrable on $[a, \infty)$ for some $a > 0$. Let $C_1, \ldots, C_n$ be as in \eqref{E:Ci}, and put $\C = (C_1, \ldots, C_n)'$. If $f \in \Pi(\g)$ with $\g$-index $\c$, then there exists an upper triangular matrix $\Q \in \RR^{n \times n}$ such that
\begin{equation}
\label{E:g2C:n}
	\C(t) = \Q \g(t) + o(g_n(t)).
\end{equation}
The matrix $\Q$ is non-singular if and only if $c_1 \prod_{i=1}^{n-1} B_{i,i+1} \neq 0$. In that case, $\C$ is a rate vector itself, $\C \in \RV_{\K_n}$ with $\K_n$ as in \eqref{E:Kn}, and $f \in \Pi(\C)$ with $\C$-index $(1, \ldots, 1)$.
\end{theorem}

Hence, for $\Pi$-varying functions $f$ of order $n$, iterating the indices transform $n$ times provides a method to construct an appropriate rate vector directly out of $f$, the index matrix and index vector being of standard form. For smooth functions, however, a much easier procedure leading to an even simpler index matrix and index vector will be described in Section~\ref{S:PiII}.

\begin{proof}
Let $\I_k$ be the $k \times k$ identity matrix and let $\B_k = (B_{ij})_{i,j = k}^n$ be the lower-right square submatrix of $\B$ starting at $B_{kk}$. Define row vectors $\c_0, \ldots, \c_{n-1}$ and $\q_1, \ldots, \q_n$ recursively as follows:
\begin{equation}
\label{E:Q}
	\begin{array}{rcl}
	\c_0 &=& \c \in \RR^{1 \times n}, \\ [1ex]
	\q_i &=& \c_{i-1} (\I_{n-i+1} + \B_{i})^{-1} \in \RR^{1 \times (n-i+1)}, 
		\qquad i \in \{1, \ldots, n\}, \\ [1ex]
	\c_i &=& \bigl( ( \q_i \B_{i} )_j \bigr)_{j = i+1}^n \in \RR^{1 \times (n-i)},
	\qquad i \in \{1, \ldots, n-1\}.
	\end{array}
\end{equation}
From Theorem~\ref{T:Pi:repr}, we obtain by induction that
\[
	C_i(t) = \q_i \bigl( g_i(t), \ldots, g_n(t) \bigr)' + o(g_n(t)),
	\qquad i \in \{1, \ldots, n\}
\]
and that
\[
	C_i \in \Pi \bigl( (g_{i+1}, \ldots, g_n)' \bigr) \text{ with index } \c_i,
	\qquad i \in \{0, \ldots, n-1\}.
\]
As a consequence, \eqref{E:g2C:n} holds with $(Q_{ij})_{j=i}^n = \q_i$ for $i \in \{1, \ldots, n\}$.

The recursive equations for $\c_i$ and $\q_i$ imply that
\begin{align*}
	Q_{11} &= q_{11} = c_{01} = c_1, \\
	Q_{ii} &= q_{i1} = c_{i-1,1} = q_{i-1,1} B_{i-1,i}, \qquad i \in \{2, \ldots, n\}.
\end{align*}
As a consequence, the diagonal elements of $\Q$ are all non-zero as soon as $c_1$ and $B_{i,i+1}$ are non-zero for all $i \in \{1, \ldots, n-1\}$. Since $\Q$ is upper triangular, this is equivalent to $\Q$ being non-singular. In this case, $\C$ is a regularly varying rate vector with index matrix $\Q \B \Q^{-1}$ and $f \in \Pi(\C)$ with index $\c \Q^{-1}$; see Remark~\ref{R:rv:char}(a) and Remark~\ref{R:grv:char}.

Now assume that $\Q$ is non-singular. The proof that $\Q \B \Q^{-1} = \K_n$ and $\c \Q^{-1} = (1, \ldots, 1)$ is by induction on $n$. If $n = 1$, this is clear, for $\B = \K_1 = 0$ and $\Q = 1/c_1$. So assume $n \geq 2$. Since $C_1 \in \Pi((g_2, \ldots, g_n)')$ with index $\c_1$, we can apply the induction hypothesis on $C_1$. As a result, $(C_2, \ldots, C_n)'$ is a rate vector of length $n-1$, regularly varying with index matrix $\K_{n-1}$, and $C_1 \in \Pi((C_2, \ldots, C_n)')$ with index $(1, \ldots, 1) \in \RR^{1 \times (n-1)}$. But this implies that $\C \in \RV_{\K_n}$. From the fact that $f \in \Pi(\C)$ with $\C$-index $\c \Q^{-1}$, we get by Theorem~\ref{T:Pi:repr},
\[
	C_1(t) = (\c \Q^{-1}) (\I + \K_n)^{-1} \C(t) + o(C_n(t)).
\]
Hence $(\c \Q^{-1}) (\I + \K_n)^{-1} = (1, 0, \ldots, 0)$, and thus
\[
	\c \Q^{-1} = (1, 0, \ldots, 0) (\I + \K_n) = (1, \ldots, 1),
\]
as required.
\end{proof}

\begin{theorem}[(Characterisation theorem for $\Pi$ of order $n$)]
\label{T:PiI:char}
Let $f$ be a measurable, locally integrable function defined in a neighbourhood of infinity. Define $C_1, \ldots, C_n$ as in \eqref{E:Ci} and put $\C = (C_1, \ldots, C_n)'$. The following four statements are equivalent:
\begin{flushenumerate}
\item[(i)] $C_n$ is eventually of constant sign and $|C_n|$ is slowly varying.
\item[(ii)] $\C$ is a regularly varying rate vector with index matrix $\K_n$ in \eqref{E:Kn} and $f \in \Pi(\C)$ with $\C$-index $(1, \ldots, 1)$.
\item[(iii)] There exists an index matrix $\B \in \RR^{n \times n}$ with zero diagonal and a rate vector $\g \in \RV_\B$ such that $f \in \Pi(\g)$ with $\g$-index $\c \in \RR^{1 \times n}$ and $c_1 \prod_{i=1}^{n-1} B_{i,i+1} \neq 0$.
\item[(iv)] There exists an index matrix $\B \in \RR^{n \times n}$ with zero diagonal and a rate vector $\g \in \RV_\B$ such that $f \in \Pi(\g)$ with linearly independent limit functions $h_1, \ldots, h_n$.
\end{flushenumerate}
In this case, for every rate vector $\g$ as in (iii) or (iv) there exists an upper triangular, invertible matrix $\bm{P} \in \RR^{n \times n}$ such that $\g(t) = \bm{P} \C(t) + o(C_n(t))$.
\end{theorem}

\begin{proof}
To prove the final statement: just take $\bm{P} = \bm{Q}^{-1}$, with $\bm{Q}$ as in Theorem~\ref{T:Pi:repr:n}.

\emph{(i) implies (iii).} For a real-valued function $C$ which is defined and locally integrable on $[a, \infty)$, with $a > 0$, define a new function $TC$ by
\[
	TC(t) = \int_a^t C(u) \frac{\du}{u}, \qquad t \in [a, \infty).
\]
Put $T^0 C = C$ and $T^k C = T (T^{k-1} C)$ for integer $k \geq 1$. By induction, it follows that
\[
	T^k C(t) = \int_a^t \frac{(\log t - \log u)^{k-1}}{(k-1)!} C(u) \frac{\du}{u},
	\qquad t \in [a, \infty), \, k \in \{1, 2, \ldots\}.
\]
An elementary computation shows that for integer $k \geq 1$ and for positive $x$ and $t$ such that $t, xt \in [a, \infty)$ we have
\[
	T^k C(xt) = \sum_{j=0}^{k-1} \frac{(\log x)^j}{j!} T^{k-j} C(t)
	+ \sum_{j=0}^{k-1} \frac{(\log x)^j}{j!} \int_1^x \frac{(- \log u)^{k-j-1}}{(k-j-1)!} C(ut) \frac{\du}{u}.
\]
If $C$ is eventually of constant sign and if $|C|$ is slowly varying, then by the Uniform convergence theorem for slowly varying functions,
\[
	\left( T^k C(xt) - \sum_{j=0}^{k-1} \frac{(\log x)^j}{j!} T^{k-j} C(t) \right) \Bigg/ C(t)
	\to \sum_{j=0}^{k-1} \frac{(\log x)^j}{j!} \int_1^x \frac{(- \log u)^{k-j-1}}{(k-j-1)!} \frac{\du}{u} 
	= \frac{(\log x)^k}{k!},
\]
whence
\begin{equation}
\label{E:TCk}
	T^k C(xt) = \sum_{j=0}^k \frac{(\log x)^j}{j!} T^{k-j} C(t) + o(C(t)).
\end{equation}
By \citet[Theorem~3.7.4]{BGT} and induction, it follows that each $T^k C$ is eventually of constant sign, with slowly varying absolute value, and $T^k C(t) = o(T^{k+1}C(t))$. As a consequence, $\bm{T}_n = (T^{n-1}C, T^{n-2}C, \ldots, C)'$ is an $n$-dimensional rate vector, regularly varying with index matrix $\J_n$, and $T^n C \in \Pi(\bm{T}_n)$ with $\bm{T}_n$-index $(1, 0, \ldots, 0)$.

By \eqref{E:C2f}, $C_{i-1}$ can be expressed in terms of $C_i$ by
\[
	C_{i-1}(t) = C_i(t) + \int_a^t C_i(u) \frac{\du}{u} = (I + T)C_i(t), \qquad i \in \{2, \ldots, n\},
\]
where $I$ is the identity operator. By induction, for $i \in \{0, \ldots, n\}$,
\[
	C_i(t) = (I + T)^{n-i} C_n(t) = \sum_{j=0}^{n-i} \binom{n-i}{j} T^j C_n(t).
\]
Specializing this to $i = 0$ yields
\begin{equation}
\label{E:Cn2f}
	f(t) = \sum_{j=0}^n \binom{n}{j} T^j C_n(t).
\end{equation}

Now assume that $C_n$ is eventually of constant sign and that $|C_n|$ is slowly varying. Combine \eqref{E:TCk} applied to $C_n$ with \eqref{E:Cn2f} to find, after some algebra,
\[
	f(xt) - f(t) = \sum_{s = 0}^{n-1} \sum_{j = s+1}^n \binom{n}{j} \frac{(\log x)^{j-s}}{(j-s)!} T^s C_n(t)
	+ o(C_n(t)), \qquad x \in (0, \infty).
\]
It follows that $f \in \Pi(\bm{T}_n)$ with $\bm{T}_n$-index $\left( \binom{n}{0}, \binom{n}{1}, \ldots, \binom{n}{n-1} \right)$.

\emph{(iii) implies (ii).} See Theorem~\ref{T:Pi:repr:n}.

\emph{(ii) implies (i).} Trivial.

\emph{(iii) and (iv) are equivalent.} This follows from the equivalence of (i) and (iv) in Proposition~\ref{P:Jordan} in combination with Remark~\ref{R:B2h}(b).
\end{proof}

\subsection{Differencing}
\label{SS:PiI:diff}

If $f$ is in de Haan's class $\Pi$ ($n = 1$), then a possible choice for the auxiliary function is $g(t) = \Delta f(t) = f(et) - f(t)$, with $e = \exp(1)$; see for instance de Haan's theorem \citep[Theorem~3.7.3]{BGT}. For general order $n$, one may hope that iterating this procedure yields a suitable rate vector for $f$. This provides an alternative to the iterated indices transform in Theorem~\ref{T:Pi:repr:n}.

Formally, put
\begin{align}
	\Delta^0 f(t) &= f(t), \nonumber \\
	\Delta^k f(t) &= \Delta (\Delta^{k-1} f)(t) = \Delta^{k-1} f(et) - \Delta^{k-1} f(t), 
	\qquad k \in \{1, 3, \ldots\}. \label{E:Deltak}
\end{align}
By induction, it follows that $\Delta^k f(t) = \sum_{i=0}^k \binom{k}{i} (-1)^{k-i} f(e^i t)$. For real $z$ and integer $k \geq 1$, put
\begin{align*}
	\binom{z}{0} &= 1, &
	\binom{z}{k} &= \frac{z (z-1) \cdots (z-k+1)}{k (k-1) \cdots 1} = \frac{(z)_k}{k!}.
\end{align*}

\begin{theorem}[(Higher-order differencing)]
Let $\g \in \RV_\B$ where $\B \in \RR^{n \times n}$ has zero diagonal. Let $f \in \Pi(\g)$ with $\g$-index $\c$. Write $\bm{\Delta} = (\Delta f, \Delta^2 f, \ldots, \Delta^n f)'$ with $\Delta^k f$ as in \eqref{E:Deltak}. 
\begin{flushenumerate}
\item[(a)]
There exists an upper triangular matrix $\Q \in \RR^{n \times n}$ such that
\begin{equation}
\label{E:g2Delta}
	\bm{\Delta}(t) = \Q \g(t) + o(g_n(t)).
\end{equation}
\item[(b)]
For $i \in \{0, 1, \ldots, n\}$ and $x \in (0, \infty)$,
\begin{align}
\label{E:Delta:GRV}
	\Delta^i f(xt) &= \sum_{j=i}^n \binom{\log x}{j-i} \Delta^j f(t) + o(g_n(t)).
\end{align}
\item[(c)]
The matrix $\Q$ in \emph{(a)} is invertible if and only if $c_1 \prod_{i=1}^{n-1} B_{i,i+1} \neq 0$. In this case, $\bm{\Delta}$ is a regularly varying rate vector with index matrix $\bm{M} \in \RR^{n \times n}$, which is a triangular Toeplitz matrix whose non-zero elements are given by 
\begin{equation}
\label{E:Delta:M}
	M_{ij} = \frac{(-1)^{j-i-1}}{j-i}, \qquad 1 \leq i < j \leq n,
\end{equation}
and $f \in \Pi(\bm{\Delta})$ with $\bm{\Delta}$-index $(1, -1/2, 1/3, \ldots, (-1)^{n-1}/n)$.
\end{flushenumerate}
\end{theorem}

Equation~\eqref{E:g2Delta} may be seen as related to the special case $\x = (e, e^2, \ldots, e^n)'$ of \eqref{E:f2g}. The condition in (c) that $c_1 \prod_{i = 1}^{n-1} B_{i,i+1} \neq 0$ is entirely natural in view of the Characterisation theorem for $\Pi$, see Theorem~\ref{T:PiI:char}(iii).

\begin{proof}
\emph{(a)} We have $f(xt) - f(t) = \h(x) \g(t) + o(g_n(t))$ with $\h(x) = \c \int_1^x u^{\B} u^{-1} \du$. By the substitution $\log(u) = y$, we find 
\begin{equation}
\label{E:B2he}
	\h(e) = \c \int_0^1 \exp(y\B) \dy. 
\end{equation}

We claim that for integer $k \geq 1$,
\begin{equation}
\label{E:g2Deltak}
	\Delta^k f(t) = \h(e) \{\exp(\B) - \I\}^{k-1} \g(t) + o(g_n(t)).
\end{equation}
The proof is by induction on $k$. For $k = 1$, this is just $\Delta f(t) = f(et) - f(t) = \h(e) \g(t) + o(g_n(t))$. For integer $k \geq 2$, by regular variation of $\g$ and by the induction hypothesis,
\begin{align*}
	\Delta^k f(t)
	&= \Delta^{k-1} f(et) - \Delta^{k-1} f(t) \\
	&= \h(e) \{\exp(\B) - \I\}^{k-2} \{ \g(et) - \g(t) \} + o(g_n(t)) \\
	&= \h(e) \{\exp(\B) - \I\}^{k-1} \g(t) + o(g_n(t)),
\end{align*}
as required. 

From \eqref{E:Bh} at $x = e$ we get $\B \int_0^1 \exp(y\B) \dy = \int_0^1 \exp(y\B) \dy \B = \exp(\B) - \I$. By \eqref{E:B2he} and \eqref{E:g2Deltak}, it follows that for integer $k \geq 1$,
\begin{equation}
\label{E:diff:g2f}
	\Delta^k f(t)
	= \c \B^{k-1} \Bigl( \int_0^1 \exp(y\B) \dy \Bigr)^k \g(t) + o(g_n(t)).
\end{equation}
The matrix $\int_0^1 \exp(y\B) \dy$ is upper triangular with unit diagonal. Hence the same is true for its powers. For $l \in \{1, \ldots, n-1\}$, the matrix $B^l$ is upper triangular, its diagonal and $l-1$ first superdiagonals being zero. Equation~\eqref{E:g2Delta} follows.

As a complement to \eqref{E:g2Deltak}, note that $\B^n = \0$ implies $\Delta^k f(t) = o(g_n(t))$ if $k \geq n+1$.

\emph{(b)} 
For $i \in \{1, \ldots, n\}$ and $x \in (0, \infty)$, we have by regular variation of $\g$, by \eqref{E:g2Deltak} and by \eqref{E:cute:1} in Lemma~\ref{L:cute} below,
\begin{align*}
	\Delta^i f(xt) - \Delta^i f(t) 
	&= \h(e) \{ \exp(\B) - \I \}^{i-1} (x^\B - \I) \g(t) + o(g_n(t)) \\
	&= \sum_{k=1}^{n-1} \binom{\log x}{k} \h(e) \{ \exp(\B) - \I \}^{i+k-1} \g(t) + o(g_n(t)) \\
	&= \sum_{j=i+1}^{i+n-1} \binom{\log x}{j-i} \h(e) \{ \exp(\B) - \I\}^{j-1} \g(t) + o(g_n(t)) \\
	&= \sum_{j=i+1}^n \binom{\log x}{j-i} \Delta^j f(t) + o(g_n(t)).
\end{align*}
In the last step we used the fact that $\{\exp(\B) - \I\}^n = \0$, which follows in turn from the fact that $\exp(\B) - \I = \B + \frac{1}{2!}\B^2 + \cdots + \frac{1}{(n-1)!} \B^{n-1}$ is upper triangular with zero diagonal.

Similarly, by \eqref{E:g2Deltak} and by \eqref{E:cute:2} in Lemma~\ref{L:cute} below, for $x \in (0, \infty)$,
\begin{align*}
	f(xt) - f(t)
	&= \c \int_1^x u^{\B} u^{-1} \du \, \g(t) + o(g_n(t)) \\
	&= \sum_{k=1}^n \binom{\log x}{k} \c \int_0^1 \exp(y\B) \dy \{ \exp(\B) - \I \}^{k-1} \g(t) 
	+ o(g_n(t)) \\
	&= \sum_{k=1}^n \binom{\log x}{k} \Delta^k f(t) + o(g_n(t)).
\end{align*}

\emph{(c)}
As in the proof of part (a), consider the matrix $\B^l$ for $l \in \{1, \ldots, n-1\}$. The element on position $(1,l+1)$ is equal to $\prod_{i=1}^{l} B_{i,i+1}$. As a result, the coefficient of $g_k$ on the right-hand side of \eqref{E:diff:g2f} is given by $Q_{11} = c_1$ and $Q_{kk} = c_1 \prod_{i=1}^{k-1} B_{i,i+1}$ for $k \in \{2, \ldots, n\}$.
In particular, $\Q$ is invertible if and only if $c_1 \prod_{i=1}^{n-1} B_{i,i+1} \neq 0$, in which case $\bm{\Delta}$ is a regularly varying rate vector with index matrix $\bm{M} = \Q \B \Q^{-1}$.

Since $\Delta^n f(t) = Q_{nn} g_n(t) + o(g_n(t))$, if $Q_{nn} \neq 0$, then in \eqref{E:Delta:GRV}, we may replace the remainder term $o(g_n(t))$ by $o(\Delta^n f(t))$. The form of the index matrix $\bm{M}$ of $\bm{\Delta}$ and of the $\bm{\Delta}$-index of $f$ follows from differentiation of $\binom{\log x}{k}$ at $x = 1$.
\end{proof}

\begin{lemma}
\label{L:cute}
For $x \in (0, \infty)$ and for upper triangular $\B \in \RR^{n \times n}$ with zero diagonal, we have
\begin{align}
\label{E:cute:1}
	x^\B - \I 
	&= \sum_{k=1}^{n-1} \binom{\log x}{k} \{ \exp(\B) - \I \}^k, \\
\label{E:cute:2}
	\int_1^x u^{\B} u^{-1} \du
	&= \int_0^1 \exp(y\B) \dy \sum_{k=1}^n \binom{\log x}{k} \{ \exp(\B) - \I \}^{k-1}.
\end{align}
\end{lemma}

\begin{proof}
In the formal series expansion $(1 + u)^v - 1 = \sum_{k=1}^\infty \binom{v}{k} u^k$, set $u = e^b - 1$ and $v = \log x$ to get
\[
	x^b - 1 = \sum_{k=1}^\infty \binom{\log x}{k} (e^b-1)^k.
\]
Since the matrix exponential function is defined through a series expansion, the identity continues to hold if we replace $b$ by a square matrix $\B \in \RR^{n \times n}$, provided the two series that result converge. If $\B$ is upper triangular with zero diagonal, then the same is true for $\exp(\B) - \I$, so that $\B^n = \{ \exp(\B) - \I \}^n = \0$, which implies that the two series are necessarily finite. This implies~\eqref{E:cute:1}

The proof of~\eqref{E:cute:2} is similar, starting from the identity
\[
	\int_1^x u^{b-1} \du
	= \int_0^1 e^{by} \dy \, \sum_{k=1}^\infty \binom{\log x}{k} (e^b-1)^{k-1}.
\]
Again, replace $b$ by $\B$ and note that $\{\exp(\B) - 1\}^n = \0$.
\end{proof}

\section{$\Pi$-variation of arbitrary order. II. Smooth functions}
\label{S:PiII}

In general, it can be cumbersome to determine whether a particular function $f$ is in some higher-order class $\Pi$ and to determine the corresponding rate vector. However, if $f$ is $n$-times continuously differentiable, then the work can often be reduced significantly (Subsection~\ref{SS:PiII:smooth}). A nice by-product of the approach is that it always gives a rate vector whose index matrix is the Jordan block $\J_n$ in \eqref{E:Jordan}. The approach even works for functions that are given as inverses of functions in the class $\Gamma$ (Subsection~\ref{SS:PiII:Gamma}) and for which no explicit form is available. In Section~\ref{S:examples} for instance, the approach is illustrated for inverses of functions related to the complementary error and complementary gamma functions.

\subsection{Smooth functions}
\label{SS:PiII:smooth}

Let $f$ be a real-valued function defined in a neighbourhood of infinity. Suppose that $f$ is $n$-times continuously differentiable on $(t_0, \infty)$ for some $t_0 > 0$. Let $D$ denote the differential operator. Define functions $L_0, \ldots, L_n$ recursively as follows:
\begin{equation}
\label{E:Lk}
	\begin{array}{ll}
		L_0(t) &= f(t),  \\
		L_k(t) &= t DL_{k-1}(t), \qquad k \in \{1, \ldots, n\}.
	\end{array}
\end{equation}
We have $L_k = D(L_{k-1} \circ \exp) \circ \log$, and thus
\[
	L_k = D^k (f \circ \exp) \circ \log, \qquad k \in \{1, \ldots, n\}
\]
a formula which greatly simplifies the actual computation of the functions $L_k$.

\begin{lemma}
\label{L:L}
Let $f$ and $L_0, \ldots, L_n$ be as in \eqref{E:Lk}. For positive $x$ and $t$ such that $t > t_0$ and $xt > t_0$, we have
\begin{equation}
\label{E:L2f}
	f(xt) - f(t) = \sum_{k=1}^n \frac{(\log x)^k}{k!}  L_k(t) + \int_1^x \frac{(\log x - \log u)^{n-1}}{(n-1)!} \{ L_n(ut) - L_n(t) \} u^{-1} \du.
\end{equation}
\end{lemma}

\begin{proof}
The lemma can be shown by applying Taylor's theorem with integral form of the remainder term to the function $f \circ \exp$. We prefer to give a direct proof, which proceeds by induction on $n$. 

For $n = 1$ we have
\begin{align*}
	f(xt) - f(t) 
	&= \int_t^{xt} L_1(y) y^{-1} \dy = \int_1^x L(ut) u^{-1} \du \\
	&= \log(x) L_1(u) + \int_1^x \{L_1(ut) - L_1(t)\} u^{-1} \du.
\end{align*}

Let $n \geq 2$. By the induction hypothesis applied to $f$,
\[
	f(xt) - f(t) = \sum_{k=1}^{n-1} \frac{(\log x)^k}{k!}  L_k(t) + \int_1^x \frac{(\log x - \log u)^{n-2}}{(n-2)!} \{ L_{n-1}(ut) - L_{n-1}(t) \} u^{-1} \du.
\]
Further, by the induction hypothesis applied to $L_{n-1}$,
\[
	L_{n-1}(ut) - L_{n-1}(t) = \log(u) L_n(t) + \int_1^u \{L_n(vt) - L_n(t)\} v^{-1} \dv.
\]
Combine the two previous displays and apply Fubini's theorem to arrive at \eqref{E:L2f}.
\end{proof}

\begin{theorem}
\label{T:L}
Let $f$ and $L_0, \ldots, L_n$ be as in \eqref{E:Lk}. If $L_n$ is eventually of constant sign and if $|L_n|$ is slowly varying, then $\bm{L} = (L_1, \ldots, L_n)'$ is a regulary varying rate vector with index matrix the Jordan block $\J_n$ in \eqref{E:Jordan}, and $f \in \Pi(\bm{L})$ with $\bm{L}$-index $\c = (1, 0, \ldots, 0)$.
\end{theorem}

\begin{addendum}
If in addition the functions $C_1, \ldots, C_n$ are defined as in \eqref{E:Ci}, then for $i \in \{1, \ldots, n\}$,
\[
	C_i(t) = \sum_{j=i}^n (-1)^{j-i} \binom{j-1}{i-1} L_j(t) + o(L_n(t)).
\]
In particular, $C_i(t) \sim L_i(t)$.
\end{addendum}

\begin{proof}
Since $L_n$ is eventually of constant sign and since $DL_{n-1}(t) = t^{-1} L_n(t)$, we find that $L_{n-1}$ is eventually increasing or decreasing, and hence of constant sign. Lemma~\ref{L:L} applied to $L_{n-1}$ gives
\[
	L_{n-1}(xt) - L_{n-1}(t) = \log (x) L_n(t) + \int_1^x \{L_n(ut) - L_n(t)\} u^{-1} \du.
\]
Since $|L_n|$ is slowly varying, the uniform convergence theorem for slowly varying functions gives
\[
	L_{n-1}(xt) - L_{n-1}(t) = \log (x) L_n(t) + o(L_n(t)).
\]
Up to the sign, we find that $L_{n-1}$ is in the univariate class $\Pi$ with auxiliary function $|L_n|$. Hence, $|L_{n-1}|$ is slowly varying as well and $L_n(t) = o(L_{n-1}(t))$. Continuing in this way we obtain that $\bm{L}$ is a rate vector. Lemma~\ref{L:L} applied to $L_i$ ($i \in \{0, \ldots, n-1\}$) gives
\[
	L_i(xt) - L_i(t) = \sum_{k=1}^{n-i} \frac{1}{k!} (\log x)^k L_{i+k}(t) + o(L_n(t)),
\]
again by application of the uniform convergence theorem for slowly varying functions. The conclusions of the theorem follow.

To prove the addendum, it suffices to apply Theorem~\ref{T:Pi:repr:n} to the rate vector $\g = \bm{L}$ and show that the upper-triangular matrix $\Q \in \RR^{n \times n}$ is given by
\[
	Q_{ij} = (-1)^{j-i} \binom{j-1}{i-1},
	\qquad 1 \leq i \leq j \leq n.
\]
Let $\q_i = (Q_{ij})_{j = i}^n \in \RR^{1 \times (n-i+1)}$ be the non-trivial part of the $i$th row of $\Q$; so $(\q_i)_k = Q_{i,i+k-1}$ for $i \in \{1, \ldots, n\}$ and $k \in \{1, \ldots, n-i+1\}$. We have to show that
\[
	(\q_i)_k = (-1)^{k-1} \binom{i+k-2}{i-1},
	\qquad i \in \{1, \ldots, n\}, \quad k \in \{1, \ldots, n-i+1\}.
\]
We will proceed by induction on $i$. From the recursive equations \eqref{E:Q} in the proof of Theorem~\ref{T:Pi:repr:n}, we know that
\begin{align*}
	\q_1 &= (1, 0, \ldots, 0) (\I_n + \J_n)^{-1}, \\
	\q_i &= \bigl( (\q_{i-1})_k \bigr)_{k=1}^{n-i+1} (\I_{n-i+1} + \J_{n-i+1})^{-1}.
\end{align*}
Since $\J_n^n = \0$, the zero matrix, we have $(\I_n + \J_n)^{-1} = \I_n + \sum_{k=1}^{n-1} (-1)^k \J_n^k$, which is an upper-triangular matrix with
\[
	\bigl( (\I_n + \J_n)^{-1} \bigr)_{ij} = (-1)^{j-i}, 
	\qquad 1 \leq i \leq j \leq n.
\]
Hence, $\q_1$ is just the first row of $(\I_n + \J_n)^{-1}$, which is $(1, -1, 1, -1, \ldots)$, as required. Suppose now that $i \geq 2$. By the induction hypothesis, for $k \in \{1, \ldots, n-i+1\}$,
\begin{align*}
	(\q_i)_k
	&= \sum_{l=1}^k (\q_{i-1})_l \bigl( (\I_{n-i+1} + \J_{n-i+1})^{-1} \bigr)_{lk} \\
	&= \sum_{l=1}^k (-1)^{l-1} \binom{i+l-3}{i-2} (-1)^{k-l} \\
	&= (-1)^{k-1} \sum_{l=0}^{k-1} \binom{i+l-2}{i-2} 
	= (-1)^{k-1} \binom{i+k-2}{i-1},
\end{align*}
where in the last step we used a well-known identity for binomial coefficients.
\end{proof}

\subsection{Inverses of $\Gamma$-varying functions}
\label{SS:PiII:Gamma}

Recall that a positive, non-decreasing, and unbounded function $A$ defined in a neighbourhood of infinity belongs to the class $\Gamma$ if there exists a positive, measurable function $a$ such that
\[
	\frac{A(t + x a(t))}{A(t)} \to e^x, \qquad x \in \RR.
\]
The class $\Gamma$ was introduced in \citet{dHL70, dH74}; see also \citet{GdH87}, \citet[Section~3.10]{BGT}, and \citet[Section~0.4.3]{Resnick87}. Functions in $\Gamma$ are inverses of positive, non-decreasing, and unbounded functions in the class $\Pi$ and vice versa. The aim in this section is to provide a simple condition on the function $A$ such that its inverse $f$ satisfies the conditions of Theorem~\ref{T:L} and is therefore $\Pi$-varying of order $n$.

Suppose that $A$ is a positive and continuously differentiable function defined in a neighbourhood of infinity. Assume $DA > 0$ and denote $q = A / DA = 1 / D \log A$, so that for some large enough $t_0$,
\begin{equation}
\label{E:q2A}
	A(t) = A(t_0) \exp \left( \int_{t_0}^t \frac{ds}{q(s)} \right), \qquad t \in [t_0, \infty).
\end{equation}
If $q$ is itself continuously differentiable and $Dq(t) = o(1)$, then it is not hard to check that $A \in \Gamma$ with auxiliary function $q$. Up to asymptotic equivalence, this representation of functions in $\Gamma$ is general.

Suppose that $q$ is $(n-1)$-times continuously differentiable. Define functions $q_1, \ldots, q_n$ by
\begin{align}
\label{E:q}
	q_1 &= q, &
	q_i &= q Dq_{i-1}, \qquad i \in \{2, \ldots, n\}.
\end{align}

\begin{theorem}\label{T:Gamma}
Let $A > 0$ be as in \eqref{E:q2A} with $q > 0$ and with inverse function $f$. If $q$ is $(n-1)$-times continuously differentiable (integer $n \geq 1$), then the functions $L_i$ in \eqref{E:Lk} are well-defined and given by $L_i = q_i \circ f$ for $i \in \{1, \ldots, n\}$. Assume further that $Dq = o(1)$, that $D^{n-1} q$ is eventually of constant sign and that $|D^{n-1} q|$ is regularly varying of index $\alpha - n + 1$ for some $\alpha \leq 1$. If $\alpha$ is not of the form $(m-1)/m$ for some integer $m \in \{2, \ldots, n-1\}$, then $L_n$ in \eqref{E:Lk} is eventually of constant sign too and $|L_n|$ is slowly varying. As a consequence, $f \in \Pi((L_1, \ldots, L_n)')$ with index $\bm{c} = (1, 0, \ldots, 0)$.
\end{theorem}

\begin{proof}
The proof that $L_i = q_i \circ f$ is by induction on $i$. First let $i = 1$. Computing derivatives in $A(f(t)) = t$ yields $DA(f(t)) \cdot Df(t) = 1$ and thus
\[
	L_1(t) = t Df(t) = \frac{A(f(t))}{DA(f(t))} = q(f(t)).
\]
Second suppose $i \in \{2, \ldots, n\}$. By the induction hypothesis,
\begin{align*}
	L_i(t) 
	= t DL_{i-1}(t) 
	&= t D(q_{i-1} \circ f)(t) \\
	&= t Dq_{i-1}(f(t)) Df(t) 
	= q(f(t)) Dq_{i-1}(f(t)) 
	= q_i(f(t)).
\end{align*}

The assumptions on $A$ imply that $A \in \Gamma$ and thus $f \in \Pi$. Since $f$ increases to infinity, this implies $f \in \RV_0$. The stated property of $L_n$ then follows if we can show that $q_n$ is eventually of constant sign and $|q_n|$ is regularly varying. This property of $q_n$ is stated formally in Proposition~\ref{P:qi} below. Theorem~\ref{T:L} then implies that $f \in \Pi((L_1, \ldots, L_n)')$ with index $\bm{c} = (1, 0, \ldots, 0)$.
\end{proof}

\begin{proposition}\label{P:qi}
Let $n \geq 2$ and let $q$ be an $(n-1)$-times continuously differentiable function defined in a neighbourhood of infinity such that $Dq(t) = o(1)$. Define $q_1 = q$ and $q_i = q D q_{i-1}$ for $i \in \{2, \ldots, n\}$. Assume that $D^{n-1} q$ is eventually of constant sign and that $|D^{n-1} q| \in \RV_{\alpha - n + 1}$ for some $\alpha \in (-\infty, 1]$. If $\alpha$ is not of the form $(m-1)/m$ for some integer $m \in \{2, \ldots, n-1\}$, then all the $q_i$ are eventually of constant sign, and $|q_i|$ is regularly varying.
\end{proposition}

\begin{proof}
The proof is presented in Lemma~\ref{L:Djq} and Propositions~\ref{P:Djqi:1}--\ref{P:Djqi:4}, with details on the asymptotic behaviour of $D^j q_i$ for all $i \in \{1, \ldots, n\}$ and $j \in \{0, \ldots, n-i\}$.
\end{proof}

\begin{example}
The restriction on $\alpha$ in Theorem~\ref{T:Gamma} and Proposition~\ref{P:qi} is unavoidable. For instance, if $q(t) = 2 t^{1/2}$, then all derivatives of $q$ are of constant sign as well as regularly varying in absolute value (with $\alpha = 1/2$), but $q_2 = q Dq = 1$ and $q_i = 0$ for $i \geq 3$. The corresponding functions $A$ and $f$ are $A(t) = \exp(t^{1/2})$ and $f(t) = (\log t)^2$. We have $f(xt) = f(t) + 2 \log(t) \log(x) + \{\log(x)\}^2$, so that $\Pi$-variation for $f$ effectively stops at order $n = 2$.
\end{example}

For real $x$ and integer $k \geq 0$, let $(x)_k$ represent the falling factorial, i.e.\
\[
	(x)_k =
	\begin{cases}
		1 & \text{if $k = 0$,} \\
		x(x-1)\cdots(x-k+1) & \text{if $k \in \{1, 2, \ldots\}$.}
	\end{cases}
\]

\begin{lemma}
\label{L:Djq}
Let $q$ be a $k$-times continuously differentiable function defined in a neighbourhood of infinity such that $Dq(t) = o(1)$. If $D^k q$ is eventually of constant sign and if $|D^k q| \in \RV_{\alpha - k}$ for some $\alpha < 1$, then for every $j \in \{1, \ldots, k\}$ the function $D^j q$ is eventually of constant sign, $|D^j q| \in \RV_{\alpha-j}$, and
\begin{equation}
\label{E:Djq}
	D^j q(t) \sim (\alpha-1)_{j-1} t^{-j+1} Dq(t).
\end{equation}
\end{lemma}

\begin{proof}
We proceed by induction on $k$. If $k = 1$, there is nothing to prove. So suppose $k \geq 2$. Then $\alpha - k < 1 - 2 = -1$. Hence $D^k q$ is integrable in a neighbourhood of infinity, and thus $D^{k-1} q(t) \to c \in \RR$. Necessarily $c = 0$, for otherwise, by successive applications of Karamata's theorem, we would have $D^{k-2} q(t) \sim ct$, $D^{k-3} q(t) \sim ct^2/2$, \ldots, $Dq(t) \sim c t^{k-2}/(k-2)!$, in contradiction to the assumption that $Dq(t) = o(1)$. Hence $D^{k-1} q(t) \to c = 0$, and thus, by Karamata's theorem,
\[
	D^{k-1} q(t) = - \int_t^\infty D^k q \sim \frac{t D^k q(t)}{\alpha - k + 1}.
\]
As a consequence, $D^{k-1} q$ is eventually of constant sign and $|D^{k-1} q| \in \RV_{\alpha - k + 1}$. The induction hypothesis yields \eqref{E:Djq} for $j \in \{1, \ldots, k-1\}$. The case $j = k$ follows from
\begin{align*}
	D^k q(t) 
	&\sim (\alpha - k + 1) t^{-1} D^{k-1} q(t) \\
	&\sim \{(\alpha-1) - (k-1) + 1\} t^{-1} (\alpha-1)_{k-2} t^{-k+2} D q(t) \\
	&= (\alpha-1)_{k-1} t^{-k+1} D q(t).
\end{align*}
This concludes the proof of the lemma.
\end{proof}

In a relation of the form $a(t) \sim c \, b(t)$, if $c = 0$, then the meaning is that $a(t) = o(b(t))$. 

\begin{proposition}[(Proposition~\ref{P:qi}, case $0 < \alpha < 1$)]
\label{P:Djqi:1}
Let $q$ be an $(n-1)$-times continuously differentiable function defined in a neighbourhood of infinity such that $Dq(t) = o(1)$. Define $q_1 = q$ and $q_i = q D q_{i-1}$ for $i \in \{2, \ldots, n\}$. If $D^{n-1} q$ is eventually of constant sign and if $|D^{n-1} q| \in \RV_{\alpha - n + 1}$ for some $\alpha \in (0, 1)$, then $q$ is eventually of constant sign, $|q| \in \RV_\alpha$, and for all $i \in \{1, \ldots, n\}$ and all $j \in \{0, \ldots, n-i\}$,
\begin{equation}
\label{E:Djqi:1}
	D^j q_i(t) \sim C_i(j) t^{-i-j+1} \{q(t)\}^i
\end{equation}
with $C_i(j)$ being given recursively by
\begin{align}
\label{E:Cij:1}
	C_1(j) &= (\alpha)_j, &
	C_i(j) &= \sum_{k=0}^j C_1(k) C_{i-1}(j-k+1), \qquad i \in \{2, \ldots, n\}.
\end{align}
In particular, $C_i(0) = \prod_{k=1}^{i-1} (k \alpha - k + 1)$ for $i \in \{1, \ldots, n\}$.
\end{proposition}

Note that $C_i(0) = 0$ if $i \geq 3$ and $\alpha = (m-1) / m$ for some integer $m \in \{2, \ldots, i-1\}$.

\begin{proof}
We proceed by induction on $i$.

First, suppose $i = 1$. Lemma~\ref{L:Djq} applies, yielding \eqref{E:Djq}. Since $Dq$ is eventually of constant sign and since $|Dq|$ is regularly varying of index $\alpha - 1 > -1$, Karamata's theorem yields $q(t) \sim \alpha^{-1} t Dq(t)$. Hence $q$ is eventually of constant sign too and $|q| \in \RV_\alpha$. Moreover, $Dq(t) \sim \alpha t^{-1} q(t)$, which in combination with \eqref{E:Djq} gives 
\[
	D^j q(t) \sim (\alpha-1)_{j-1} t^{-j+1} \alpha t^{-1} q(t)
	= (\alpha)_j t^{-j} q(t),
\]
which is \eqref{E:Djqi:1} for $i = 1$.

Second, suppose $i \in \{2, \ldots, n\}$. Let $j \in \{0, \ldots n-i\}$. Apply Leibniz' product rule to see that
\[
	D^j q_i = D^j (q Dq_{i-1}) = \sum_{k=0}^j \binom{j}{k} (D^k q) (D^{j-k+1} q_{i-1}).
\]
By the induction hypothesis, for $k \in \{0, \ldots, j\}$,
\begin{align*}
	D^k q(t) \cdot D^{j-k+1} q_{i-1}(t)
	&\sim C_1(k) t^{-k} q(t) \cdot C_{i-1}(j-k+1) t^{-(j-k+1)-(i-1)+1} \{q(t)\}^{i-1} \\
	&= C_1(k) C_{i-1}(j-k+1) t^{-i-j+1} \{q(t)\}^i.
\end{align*}
Sum over all $k$ to arrive at \eqref{E:Djqi:1} with $C_i(j)$ as in \eqref{E:Cij:1}.

To arrive at the stated expression for $C_i(0)$, proceed as follows. Define functions $\Psi_i$ and coefficients $C_i^*(j)$ by
\[
	\Psi_i(z) 
	= \biggl(\prod_{k=1}^{i-1} (k\alpha - k + 1)\biggr)(1 + z)^{i\alpha - i + 1}
	= \sum_{j=0}^\infty \frac{C_i^*(j)}{j!} z^j, \qquad -1 < z < 1.
\]
Then $\Psi_1(z) = (1+z)^\alpha$ so $C_1^*(j) = (\alpha)_j = C_1(j)$. Moreover, it is not hard to see that $\Psi_i(z) = \Psi_1(z) \, D\Psi_{i-1}(z)$ for $i \geq 2$. Comparing coefficients in the corresponding power series yields $C_i^*(j) = \sum_{k=0}^j C_1^*(k) C_{i-1}^*(j-k+1)$, which is the same recursion relation as in \eqref{E:Cij:1}. As a consequence, $C_i^*(j) = C_i(j)$ for all $i$ and $j$. But then $C_i(0) = C_i^*(0) = \Psi_i(0) = \prod_{k=1}^{i-1} (k\alpha - k + 1)$.
\end{proof}

\begin{proposition}[(Proposition~\ref{P:qi}, case $\alpha = 0$)]
\label{P:Djqi:2}
Let $q$ be an $(n-1)$-times continuously differentiable function defined in a neighbourhood of infinity such that $Dq(t) = o(1)$. Define $q_1 = q$ and $q_i = q D q_{i-1}$ for $i \in \{2, \ldots, n\}$. If $D^{n-1} q$ is eventually of constant sign and if $|D^{n-1} q| \in \RV_{-n+1}$, then $q \in \Pi$ with auxiliary function $t \mapsto t Dq(t)$, and for all $i \in \{1, \ldots, n\}$ and all $j \in \{0, \ldots, n-i\}$ such that $i + j \geq 2$,
\begin{equation}
\label{E:Djqi:2}
	D^j q_i(t) \sim (-1)^{i+j} (i+j-2)! t^{-i-j+2} \{q(t)\}^{i-1} Dq(t).
\end{equation}
\end{proposition}

\begin{proof}
By Lemma~\ref{L:Djq}, the function $Dq$ is eventually of constant sign and $|Dq| \in \RV_{-1}$. Hence $q \in \Pi$ with auxiliary function $t \mapsto t Dq(t)$. To show \eqref{E:Djqi:2}, we proceed by induction on $i$. 

First, suppose $i = 1$. Then the statement is that for all $j \in \{1, \ldots, n\}$ we have $D^j q(t) \sim (-1)^{j+1} (j-1)! t^{-j+1} Dq(t)$. But as $(-1)^{j+1} (j-1)! = (-1)_{j-1}$, this is just \eqref{E:Djq} in Lemma~\ref{L:Djq}.

Second, suppose $i \in \{2, \ldots, n\}$. Then the range of $j$ is $\{0, \ldots, n-i+1\}$. By Leibniz' product rule,
\[
	D^j q_i = D^j (q Dq_{i-1}) = \sum_{k=0}^j \binom{j}{k} (D^k q) (D^{j-k+1} q_{i-1}).
\]
By the induction hypothesis, the term corresponding to $k = 0$ is
\begin{align*}
	q(t) D^{j+1} q_{i-1}(t) 
	&\sim q(t) (-1)^{(i-1)+(j+1)} \{(i-1)+(j+1)-2\}! t^{-(i-1)-(j+1)+2} \{q(t)\}^{i-2} Dq(t) \\
	&= (-1)^{i+j} (i+j-2)! t^{-i-j+2} \{q(t)\}^{i-1} Dq(t).
\end{align*}
The terms corresponding to $k \in \{1, \ldots, j\}$ are of smaller order: by the induction hypothesis,
\begin{align*}
	D^k q(t) \cdot D^{j-k+1} q_{i-1}(t)
	&= O \bigl( t^{-k+1} Dq(t) \bigr) O \bigl( t^{-(j-k+1)-(i-1)+2} \{q(t)\}^{i-2} Dq(t) \bigr) \\
	&= O \bigl( t^{-i-j+2} \{q(t)\}^{i-1} Dq(t) \bigl) O \biggl( \frac{tDq(t)}{q(t)} \biggr).
\end{align*}
Since $t Dq(t) = o(q(t))$, indeed $D^k q(t) \cdot D^{j-k+1} q_{i-1}(t) = o\bigl(t^{-i-j+2} \{q(t)\}^{i-1} Dq(t)\bigr)$.
\end{proof}

\begin{proposition}[(Proposition~\ref{P:qi}, case $\alpha < 0$)]
\label{P:Djqi:3}
Let $q$ be an $(n-1)$-times continuously differentiable function defined in a neighbourhood of infinity such that $Dq(t) = o(1)$. Define $q_1 = q$ and $q_i = q D q_{i-1}$ for $i \in \{2, \ldots, n\}$. If $D^{n-1} q$ is eventually of constant sign and if $|D^{n-1} q| \in \RV_{\alpha - n + 1}$ for some $\alpha \in (-\infty, 0)$, then $q(\infty) = \lim_{t \to \infty} q(t)$ exists in $\RR$, the function $q - q(\infty)$ is eventually of constant sign, $|q - q(\infty)| \in \RV_\alpha$, and now:
\begin{itemize}
\item[(i)] If $q(\infty) = 0$, then \eqref{E:Djqi:1} and \eqref{E:Cij:1} in Proposition~\ref{P:Djqi:1} are valid.
\item[(ii)] If $q(\infty) \neq 0$, then for all $i \in \{1, \ldots, n\}$ and all $j \in \{0, \ldots, n-i\}$ such that $i + j \geq 2$,
\begin{equation}
\label{E:Djqi:3}
	D^j q_i(t) \sim (\alpha-1)_{i+j-2} \{q(\infty)\}^{i-1} t^{-i-j+2} Dq(t).
\end{equation}
\end{itemize}
\end{proposition}

Note that \eqref{E:Djqi:2} can be viewed as the case $\alpha = 0$ of \eqref{E:Djqi:3}.

\begin{proof}
By Lemma~\ref{L:Djq}, $Dq$ is eventually of constant sign and $|Dq| \in \RV_{\alpha-1}$. Since $\alpha-1 < -1$, the function $Dq$ is integrable in a neighbourhood of infinity. Hence $q(t) \to q(\infty) \in \RR$, and by Karamata's theorem,
\[
	q(t) - q(\infty) = - \int_t^\infty Dq \sim \frac{t Dq(t)}{\alpha}.
\]
As a consequence, $q - q(\infty)$ is eventually of constant sign too and $|q - q(\infty)| \in \RV_\alpha$.

\textit{(i)} If $q(\infty) = 0$, then the above display becomes $Dq(t) \sim \alpha t^{-1} q(t)$ and we can proceed in exactly the same way as in Proposition~\ref{P:Djqi:1}.

\textit{(ii)} Suppose $q(\infty) \neq 0$. To show \eqref{E:Djqi:3}, we proceed by induction on $i$.

First, if $i = 1$, then the statement is that for all $j \in \{1, \ldots, n\}$ we have $D^j q(t) \sim (\alpha-1)_{j-1} t^{-j+1} Dq(t)$. But this is just \eqref{E:Djq} in Lemma~\ref{L:Djq}.

Second, suppose $i \in \{2, \ldots, n\}$. Then the range of $j$ is $\{0, \ldots, n-i+1\}$. By Leibniz' product rule,
\[
	D^j q_i = D^j (q Dq_{i-1}) = \sum_{k=0}^j \binom{j}{k} (D^k q) (D^{j-k+1} q_{i-1}).
\]
By the induction hypothesis, the term corresponding to $k = 0$ is
\begin{align*}
	q(t) D^{j+1} q_{i-1}(t) 
	&\sim q(\infty) (\alpha-1)_{(j+1)+(i-1)-2} t^{-(i-1)-(j+1)+2} \{q(\infty)\}^{i-2} Dq(t) \\
	&= (\alpha-1)_{j+i-2} t^{-i-j+2} \{q(\infty)\}^{i-1} Dq(t).
\end{align*}
The terms corresponding to $k \in \{1, \ldots, j\}$ are of smaller order: by the induction hypothesis,
\begin{align*}
	D^k q(t) \cdot D^{j-k+1} q_{i-1}(t)
	&= O \bigl( t^{-k+1} Dq(t) \bigr) O \bigl( t^{-(j-k+1)-(i-1)+2} Dq(t) \bigr) \\
	&= O \bigl( t^{-i-j+2} Dq(t) \bigr) O\bigl( tDq(t) \bigr).
\end{align*}
Since $t Dq(t) = o(1)$, indeed $D^k q(t) \cdot D^{j-k+1} q_{i-1}(t) = o\bigl(t^{-i-j+2} Dq(t)\bigr)$.
\end{proof}

\begin{proposition}[(Proposition~\ref{P:qi}, case $\alpha = 1$)]
\label{P:Djqi:4}
Let $q$ be an $(n-1)$-times continuously differentiable function defined in a neighbourhood of infinity such that $Dq(t) = o(1)$. Define $q_1 = q$ and $q_i = q D q_{i-1}$ for $i \in \{2, \ldots, n\}$. If $D^{n-1} q$ is eventually of constant sign and if $|D^{n-1} q| \in \RV_{- n + 2}$, then $q \in \RV_1$ and for $i \in \{1, \ldots, n\}$ and $j \in \{0, \ldots, n-i\}$,
\begin{equation}
\label{E:Djqi:4}
	D^j q_i(t) \sim 
	\begin{cases}
		t^{-i-j+1} \{q(t)\}^i & \mbox{if $j \leq 1$,} \\
		i(-1)^{j-2} (j-2)! t^{-i-j+1} \{q(t)\}^i H(t) & \mbox{if $j \geq 2$.}
	\end{cases}
\end{equation}
where $H(t) = t D^2 q(t) / Dq(t)$ is eventually of constant sign, $|H| \in \RV_0$, and $H(t) = o(1)$.
\end{proposition}

\begin{proof}
If $n = 2$ then the case $j \geq 2$ cannot occur, whereas the case $j \leq 1$ follows from Karamata's theorem, which stipulates that $q(t) \sim t Dq(t)$, yielding $Dq(t) \sim t^{-1} q(t)$ and $q_2(t) = q(t) Dq(t) \sim t^{-1} \{q(t)\}^2$.

Next assume $n \geq 3$. Then $D^2 q$ is eventually of constant sign and $|D^2 q| \in \RV_{-1}$; for $n = 3$ this follows by assumption, while if $n \geq 4$, this follows from
\begin{equation}
\label{E:Djq:4}
	D^j q(t) \sim (-1)_{j-2} t^{-j+2} D^2 q(t), \qquad j \in \{2, \ldots, n\},
\end{equation}
which can be shown in a similar way as in Lemma~\ref{L:Djq}. As a consequence, $Dq$ is in the class $\Pi$ with auxiliary function $t \mapsto t D^2 q(t)$. In particular, $Dq$ is also eventually of constant sign, $|Dq| \in \RV_0$, and $t D^2q(t) = o \bigl( Dq(t) \bigr)$. The statements about $H$ follow. By Karamata's theorem, $q(t) \sim t Dq(t)$. To show \eqref{E:Djqi:4}, we proceed by induction on $i$.

First if $i = 1$, then \eqref{E:Djqi:4} is just $q(t) \sim q(t)$ for $j = 0$, which is trivial, $Dq(t) \sim t^{-1} q(t)$ for $j = 1$, which is Karamata's theorem, and $D^j q(t) \sim (-1)^{j-2} (j-2)! t^{-j} q(t) H(t) \sim (-1)_{j-2} t^{-j+2} D^2 q(t)$ for $j \geq 2$, which is \eqref{E:Djq:4}.

Second suppose $i \geq 2$. If $j = 0$, then by the induction hypothesis,
\begin{align*}
	q_i(t) 
	&= q(t) Dq_{i-1}(t) \\
	&\sim q(t) t^{-(i-1)-1+1} \{q(t)\}^{i-1} = t^{-i+1} \{q(t)\}^i.
\end{align*}
If $j = 1$, then by the induction hypothesis and the fact that $H(t) = o(1)$,
\begin{align*}
	Dq_i(t) &= D \bigl( q(t) Dq_{i-1}(t) \bigr) 
	= Dq(t) Dq_{i-1}(t) + q(t) D^2 q_{i-1}(t) \\
	&= t^{-1} q(t) \cdot t^{-i+1} \{q(t)\}^{i-1} \{1 + o(1)\} 
	+ q(t) \cdot (i-1) t^{-(i-1)-2+1} \{q(t)\}^{i-1} H(t) \{1 + o(1)\} \\
	&\sim t^{-i} \{q(t)\}^i.
\end{align*}
If $j \geq 2$, then as before we start from Leibniz' product rule:
\[
	D^j q_i = D^j (q Dq_{i-1}) = \sum_{k=0}^j \binom{j}{k} (D^k q) (D^{j-k+1} q_{i-1}).
\]
The terms $k \in \{0, 1, j\}$ will contribute to the asymptotics, while the other ones will turn out to be asymptotically negligible.
\begin{itemize}
\item \underline{Case $k=0$:} By the induction hypothesis,
\begin{align*}
	\binom{j}{0} q(t) \cdot D^{j+1} q_{i-1}(t) 
	&\sim q(t) \cdot (i-1) (-1)^{j-1} (j-1)! t^{-i-j+1} \{q(t)\}^{i-1} H(t) \\
	&= (i-1) (-1)^{j-1} (j-1)! t^{-i-j+1} \{q(t)\}^i H(t).
\end{align*}
\item \underline{Case $k=1$:} By the induction hypothesis,
\begin{align*}
	\binom{j}{1} Dq(t) \cdot D^j q_{i-1}(t)
	&\sim j t^{-1} q(t) \cdot (i-1) (-1)^{j-2} (j-2)! t^{-(i-1)-j+1} \{q(t)\}^{i-1} H(t) \\
	&= (i-1) (-1)^{j-2} (j-2)! j t^{-i-j+1} \{q(t)\}^i H(t).
\end{align*}
\item \underline{Case $2 \leq k \leq j-1$:} By the induction hypothesis, since $k \geq 2$ and $j-k+1 \geq 2$ and since $H(t) = o(1)$,
\begin{align*}
	D^k q(t) \cdot D^{j-k+1} q_{i-1}(t)
	&= O \bigl( t^{-1-k+1} q(t) H(t) \cdot t^{-(i-1)-(j-k+1)+1} \{q(t)\}^{i-1} H(t) \bigr) \\
	&= O \bigl( t^{-i-j+1} \{q(t)\}^i H^2(t) \bigr) \\
	&= o \bigl( t^{-i-j+1} \{q(t)\}^i H(t) \bigr).
\end{align*}
\item \underline{Case $k = j$:} By the induction hypothesis,
\begin{align*}
	\binom{j}{j} D^j q(t) \cdot Dq^{i-1}(t)
	&\sim (-1)^{j-2} (j-2)! t^{-1-j+1} q(t) H(t) \cdot t^{-i+1} \{q(t)\}^{i-1} \\
	&= (i-1) (-1)^{j-2} (j-2)! t^{-i-j+1} \{q(t)\}^i H(t).
\end{align*}
\end{itemize}
Summing over all $k \in \{0, \ldots, j\}$, we get
\begin{align*}
	\lim_{t \to \infty} \frac{D^j q_i(t)}{t^{-i-j+1} \{q(t)\}^i H(t)}
	&= (i-1) (-1)^{j-1} (j-1)! + (i-1) (-1)^{j-2} (j-2)! j + (-1)^{j-2} (j-2)! \\
	&= (-1)^{j-2} (j-2)! \{ -(i-1) (j-1) + (i-1) j + 1 \} \\
	&= i (-1)^{j-2} (j-2)!
\end{align*}
as required.
\end{proof}

\section{Other special cases}
\label{S:speccas}
\subsection{All indices different from zero}
\label{SS:speccas:nonzero}

Let $\g \in \RV_\B$. If none of the diagonal elements of $\B$ are zero, that is, if none of the rate functions $g_i$ are slowly varying, then the class $\Pi(\g)$ essentially consists of linear combinations of the rate functions and the constant function.

\begin{theorem}
\label{T:speccas:nonzero}
Let $\B \in \RR^{n \times n}$ be upper triangular and invertible, i.e.\ without zeroes on the diagonal. Let $\g \in \RV_\B$. Then $f \in \GRV(\g)$ with $\g$-index $\c$ if and only if there exists a constant $C$ such that
\begin{equation}
\label{E:speccas:nonzero}
	f(t) = C + \c \B^{-1} \g(t) + o(g_n(t)).
\end{equation}
\end{theorem}

\begin{proof}
\textsl{Necessity.} Suppose $f \in \GRV(\g)$ with $\g$-index $\c$. Define $\xi(t) = f(t) - \c \B^{-1} \g(t)$. Then for $x \in (0, \infty)$, in view of \eqref{E:B2h:inv},
\begin{align*}
	\xi(xt)
	&= f(xt) - \c \B^{-1} \g(xt) \\
	&= f(t) + \c \B^{-1} (x^\B - \I) \g(t) - \c \B^{-1} x^\B \g(t) + o(g_n(t)) \\
	&= f(t) - \c \B^{-1} \g(t) + o(g_n(t)) \\
	&= \xi(t) + o(g_n(t)).
\end{align*}
Equation~\eqref{E:speccas:nonzero} now follows from the representation theorem for $o\Pi_g$ for auxiliary functions which are regularly varying but not slowly varying \citep[Theorems~3.6.1$^\pm$, pp.~152--153]{BGT}. Note that in case $b_n > 0$, the constant $C$ can be absorbed in the $o(g_n(t))$ remainder term.

\paragraph{Sufficiency.} For $f$ as in \eqref{E:speccas:nonzero}, we have
\[
	f(xt) - f(t) = \c \B^{-1} (x^\B - \I) \g(t) + o(g_n(t))
\]
so that, by \eqref{E:B2h:inv}, indeed $f \in \GRV(\g)$ with $\g$-index $\c$.
\end{proof}

\begin{remark}
\label{R:speccas:nonzero}
In the special case where all diagonal elements of $\B$ are equal to $b$, nonzero, we find
\[
	t^{-b} \{ f(t) - C \} = \c \B^{-1} \tilde{\g}(t) + o(\tilde{g}_n(t))
\]
the rate vector $\tilde{\g}(t) = t^{-b} \g(t)$ being regularly varying with index matrix $\tilde{\B} = \B - b\I$, all of whose diagonal elements are zero, a case which was covered in Sections~\ref{S:PiI}--\ref{S:PiII}.
\end{remark}

\subsection{All indices distinct}
\label{SS:speccas:diff}

In case all the diagonal elements of the index matrix $\B$ are distinct, rate vectors $\g$ in the class $\RV_\B$ are essentially given by linear combinations of power functions and $g_n$. This representation can be extended to the class $\GRV(\g)$ with $\g \in \RV_\B$. In addition, the Potter bounds for $\RV_\B$ and $\GRV(\g)$ can be sharpened.

For a square matrix $\A$, let $\diag(\A)$ be the diagonal matrix of the same dimension containing the diagonal elements of $\A$.

\begin{theorem}
\label{T:rv:repr:diff}
Let $\B \in \RR^{n \times n}$ ($n \geq 2$) be upper triangular with distinct diagonal elements $B_{11} > \cdots > B_{nn}$. Put $b_i = B_{ii}$ and $\D = \diag(\B) \in \RR^{n \times n}$. Then a rate vector $\g$ belongs to $\RV_\B$ if and only if $|g_n| \in \RV_{b_n}$ and there exists an upper triangular and invertible matrix $\Q \in \RR^{n \times n}$ such that $\Q_{nn} = 1$, $\B = \Q \D \Q^{-1}$ and
\begin{equation}
\label{E:rv:repr:diff}
	g_i(t) = \sum_{j = i}^{n-1} Q_{ij} t^{b_j} + Q_{in} g_n(t) + o(g_n(t)), \qquad i \in \{1, \ldots, n-1\}.
\end{equation}
\end{theorem}

\begin{proof}
\textsl{Necessity.}
Suppose that $\g \in \RV_\B$. By Lemma~\ref{L:speccas:diff} below, there exists an invertible, upper triangular matrix $\bm{P}$ such that $\B = \bm{P} \D \bm{P}^{-1}$ and with $P_{ii} = 1$ for all $i$. Then $\tilde{\g} := \bm{P}^{-1} \g$ is a rate vector as well and is regularly varying with index matrix $\bm{P}^{-1} \B \bm{P} = \D$, see Remark~\ref{R:rv:char}(a). For $x \in (0, \infty)$, the matrix $x^\D$ is diagonal and with diagonal elements $x^{b_1}, \ldots, x^{b_n}$. Fix $i \in \{1, \ldots, n-1\}$. Row number $i$ of the relation $\tilde{\g}(xt) = x^{\D} \tilde{\g}(t) + o(g_n(t))$ is just 
\[
	\tilde{g}_i(xt) = x^{b_i} \tilde{g}_i(t) + o(g_n(t)).
\]
Writing $L_i(t) = t^{-b_i} \tilde{g}_i(t)$, we find
\[
	L_i(xt) = L_i(t) + o(t^{-b_i} g_n(t)).
\]
Since the function $t^{-b_i} |g_n(t)|$ is regularly varying with negative index $b_n - b_i$, the representation theorem for $o\Pi_g$ for auxiliary functions $g$ with positive increase \citep[Theorem~3.6.1$^-$, p.~152]{BGT} implies that there exists $C_i \in \RR$ such that
\[
	L_i(t) = C_i + o(t^{-b_i} g_n(t)),
\]
and thus
\[
	\tilde{g}_i(t) = C_i t^{b_i} + o(g_n(t)).
\]
Since $\tilde{g}$ is a rate vector, $C_i$ must be nonzero. Write $\bm{C} = (C_1, \ldots, C_{n-1}, 1)$. Since $\g = \bm{P} \tilde{\g}$, we find
\[
	\g(t) = \bm{P} \bm{C} (t^{b_1}, \ldots, t^{b_{n-1}}, g_n(t))' + o(g_n(t)),
\]
which is \eqref{E:rv:repr:diff} in matrix notation and with $\Q = \bm{P} \bm{C}$. Finally, since diagonal matrices commute, $\Q \D \Q^{-1} = \bm{P} \bm{C} \D \bm{C}^{-1} \bm{P}^{-1} = \bm{P} \D \bm{P}^{-1} = \B$.

\paragraph{Sufficiency.}
The rate vector $\hat{\g}(t) = (t^{b_1}, \ldots, t^{b_{n-1}}, g_n(t))'$ is regularly varying with index matrix $\D$. By Remark~\ref{R:rv:char}(a), $\g = \Q \hat{\g}$ is regularly varying with index matrix $\Q \D \Q^{-1} = \B$.
\end{proof}

\begin{theorem}
\label{T:grv:repr:diff}
Let $\B \in \RR^{n \times n}$ ($n \geq 2)$ be upper triangular with distinct diagonal elements $B_{11} > \cdots > B_{nn}$. Put $b_i = B_{ii}$. Then $f \in \Pi(\g)$ for some $\g \in \RV_\B$ if and only if there exist $a, a_0, \ldots, a_n \in \RR$ such that 
\begin{equation}
\label{E:grv:repr:diff}
	f(t) = a_0 + \sum_{i=1}^{n-1} a_i \int_1^t u^{b_i - 1} \du + \int_a^t \{a_n + o(1)\} g_n(u) u^{-1} \du + o(g_n(t)).
\end{equation}
\end{theorem}

\begin{proof}
Let $\hat{\g}(t) = (t^{b_1}, \ldots, t^{b_{n-1}}, g_n(t))'$. By Theorem~\ref{T:rv:repr:diff}, $\g(t) = \Q \hat{\g}(t) + o(g_n(t))$ for some invertible, upper triangular matrix $\Q$ satisfying $Q_{nn} = 1$ and $\Q \diag(\B) \Q^{-1} = \B$.

\paragraph{Necessity.}
Suppose that $f \in \GRV(\g)$ with $\g$-index $\bm{c} \in \RR^{1 \times n}$. By the representation theorem for $\GRV(\g)$ (Theorem~\ref{T:grv:repr}), we have
\[
	f(t) = v + o(g_n(t)) + \int_a^t \{ \bm{c} \g(u) + o(g_n(u)) \} u^{-1} \du.
\]
Putting $\bm{a} = \c \Q$, we find $\c \g(t) = \bm{a} \hat{\g}(t) + o(g_n(t))$ and thus
\[
	f(t) = v + \sum_{i=1}^{n-1} a_i \int_a^t u^{b_i-1} \du + \int_a^t \{a_n + o(1)\} g_n(u) u^{-1} \du + o(g_n(t)).
\]
Put $a_0 = v - \sum_{i=1}^{n-1} a_i \int_1^a u^{b_i-1} \du$ to arrive at \eqref{E:grv:repr:diff}.

\paragraph{Sufficiency.} Put $\c = (a_1, \ldots, a_n) \Q^{-1}$, do the steps of the previous paragraph in reverse, and apply the representation theorem for $\GRV(\g)$ (Theorem~\ref{T:grv:repr:diff}).
\end{proof}

In case $B_{nn} \neq 0$, the representation in \eqref{E:grv:repr:diff} can be even further simplified to
\[
	f(t) = a_0 + \sum_{i=1}^{n-1} a_i \int_1^t u^{b_i - 1} \du + a_n g_n(t) + o(g_n(t))
\]
(where $a_0$ and $a_n$ have been redefined). Note also that for those $i \in \{1, \ldots, n-1\}$ for which $b_i$ is nonzero (which has to occur for all but at most one $i$), the integral $\int_1^t u^{b_i-1} \du$ can be replaced by $t^{b_i}$ (after redefining $a_i$ and $a_0$).

\begin{theorem}[(Improved Potter bounds for $\RV_\B$ and $\GRV(\g)$)]
\label{T:diff:Potter}
Let $\B \in \RR^{n \times n}$ with all diagonal elements distinct. Put $b_n = B_{nn}$.
\begin{flushenumerate}
\item[(a)]
If $\g \in \RV_\B$, then for every $\eps > 0$, there exists $t(\eps) > 0$ such that
\begin{equation}
\label{E:diff:Potter:rv}
	\frac{\norm{\g(xt) - x^\B \g(t)}}{|g_n(t)|} \leq \eps x^{b_n} \max(x^\eps, x^{-\eps}), \qquad t \geq t(\eps), \quad x \geq t(\eps) / t.
\end{equation}
\item[(b)]
If $\g \in \RV_\B$ and $f \in \GRV(\g)$ with limit functions $\h$, then for every $\eps > 0$, there exists $t(\eps) > 0$ such that
\begin{equation}
\label{E:diff:Potter:grv}
	\frac{|f(xt) - f(t) - \h(x) \g(t)|}{|g_n(t)|} \leq \eps \max(1, x^{b_n + \eps}, x^{b_n - \eps}) \qquad t \geq t(\eps), \quad x \geq t(\eps) / t.
\end{equation}
\end{flushenumerate}
\end{theorem}

\begin{proof}
\textsl{(a)}
Put $b_i = B_{ii}$ and $\D = \diag(\B)$. By Theorem~\ref{T:rv:repr:diff}, there exists an invertible, upper triangular matrix $\Q \in \RR^{n \times n}$ so that $\B = \Q \D \Q^{-1}$ and $\g = \Q \hat{\g}$ where $\hat{\g}(t) = (t^{b_1}, \ldots, t^{b_{n-1}}, g_n(t))'$. Since $x^\D$ is diagonal with diagonal elements $x^{b_1}, \ldots, x^{b_n}$, we have
\[
	\hat{\g}(xt) - x^\D \hat{\g}(t) = \left(0, \ldots, 0, g_n(xt) - x^{b_n} g_n(t)\right)', \qquad x \in (0, \infty).
\]
As a consequence, since $x^\B = \Q x^\D \Q^{-1}$,
\begin{align*}
	\norm{\g(xt) - x^\B \g(t)} 
	&= \norm{\Q \hat{\g}(xt) - \Q x^\D \Q^{-1} \Q \hat{\g}(t)} \\
	&\leq \norm{\Q} \, \norm{\hat{\g}(xt) - x^\D \hat{\g}(t)} \\
	&= \norm{Q} \, |g_n(xt) - x^{b_n} g_n(t)|.
\end{align*}
[here we implicitly assumed, without loss of generality, that the norm of $(0, \ldots, 0, 1)$ is equal to unity]. Now apply Potter's theorem for $\RV_{b_n}$ (Theorem~\ref{T:rv:Potter}).

\textsl{(b)}
For $x \leq 1$, this is just the general Potter bound \eqref{E:grv:Potter}. For $x \geq 1$, start again from \eqref{E:grv:Potter:10} and use \eqref{E:diff:Potter:rv} rather than the general Potter bound for $\RV_\B$.
\end{proof}

\begin{lemma}
\label{L:speccas:diff}
Let $\B \in \RR^{n \times n}$ be upper triangular and with $n$ distinct diagonal elements. Let $\D = \diag(\B) \in \RR^{n \times n}$. Then for any vector $\q \in \RR^n$ with nonzero elements there exists a unique upper triangular matrix $\Q \in \RR^{n \times n}$ with diagonal $\q$ and such that $\B = \Q \D \Q^{-1}$.
\end{lemma}

The proof of Lemma~\ref{L:speccas:diff} is similar to the one of Proposition~\ref{P:Jordan} and is therefore omitted.

\section{Examples}
\label{S:examples}

\begin{example}
The function $f(t) = \log \lfloor t \rfloor$ is in the class $\Pi$ of order $n = 1$ but not of any higher order.
\end{example}

\begin{example}[(power series)]
\label{Ex:powerseries}
Suppose that $f(x) = \sum_{k=0}^\infty a_k x^{-k \alpha}$ where $\alpha > 0$ and $\sum_{k = 0}^\infty a_k z^k$ is a power series which is convergent in a neighbourhood of $z = 0$. Then for $x > 0$ and integer $n \geq 1$,
\[
	f(xt) - f(t) = \sum_{k=1}^n a_k (x^{-k \alpha} - 1) t^{-k \alpha} + o(t^{-n \alpha}).
\]
The rate vector $\g(t) = (t^{-\alpha}, t^{-2\alpha}, \ldots, t^{-n \alpha})'$ is regularly varying with index matrix 
\[
	\B = \diag(-\alpha, -2\alpha, \ldots, -n \alpha). 
\]
By the first display, the function $f$ belongs to $\GRV(\g)$ with $\g$-index $\c = (- k a_k \alpha)_{k = 1}^n$.
\end{example}

\begin{example}[(Gamma function)]
Let $f = \log \Gamma$ with $\Gamma(t) = \int_0^\infty y^{t-1} e^{-y} \dy$ Euler's Gamma function. From \citet[\S6.1.40--42, \S23.1.1--3]{AS92}, for integer $n \geq 1$,
\[
	\log \Gamma(t) 
	= t \log(t) - t - \frac{1}{2} \log(t) + \frac{1}{2} \log(2\pi)
	+ \sum_{m=1}^n \frac{B_{2m}}{2m(2m-1)}t^{-2m+1} + O(t^{-2n-1}),
\]
with $B_k$ the Bernoulli numbers, given by $z / (e^z-1) = \sum_{k=0}^\infty B_k z^k / k!$. It follows that for fixed $x \in (0, \infty)$,
\begin{multline*}
	\log \Gamma(xt) - \log \Gamma(t)
	= t \log(t) (x-1) + t \{ x \log(x) - x + 1 \} - \frac{1}{2} \log(x) \\
	+ \sum_{m=1}^n t^{-2m + 1} \frac{B_{2m}}{2m(2m-1)} (x^{-2m+1} - 1) + O(t^{-2n-1}).
\end{multline*}
As a consequence, $\log \Gamma \in \GRV((g_1, \ldots, g_{n+2})')$ with
\[
	g_1(t) = t \log(t), \quad g_2(t) = t, 
	\quad g_3(t) = 1, \quad g_{3+m}(t) = t^{-2m+1}
\]
for integer $m \geq 1$. The non-zero elements of the index matrix $\B$ of $\g$ are
\[
	B_{11} = B_{12} = B_{22} = 1,
	\qquad B_{3+m,3+m} = -2m+1
\]
for $m \geq 1$. The $\g$-index $\c$ of $\log \Gamma$ is given by
\[
	c_1 = 1, \quad c_2 = 0, 
	\quad c_3 = - \frac{1}{2}, \quad c_{3+m} = - \frac{B_{2m}}{2m}.
\]
\end{example}

\begin{example}[(no reduction to Jordan block index matrix)]
In the following four cases, the assumption on $\B$ in Proposition~\ref{P:Jordan} is not met, so that reduction to the situation where the index matrix is a Jordan block is impossible.
\begin{flushenumerate}
\item[(a)]
If $\g(t) = (\log t, (\log t)^{1/2}, 1)'$, then $\g \in \RV_\B$ with
\[
	\B = \begin{pmatrix} 0 & 0 & 1 \\ 0 & 0 & 0 \\ 0 & 0 & 0 \end{pmatrix}.
\]
\item[(b)]
If $\g(t) = ((\log t)^{3/2}, (\log t)^{1/2}, 1)'$, then $\g \in \RV_\B$ with
\[
	\B = \begin{pmatrix} 0 & 3/2 & 0 \\ 0 & 0 & 0 \\ 0 & 0 & 0 \end{pmatrix}.
\]
\item[(c)]
If $\g(t) = (1, (\log t)^{-1}, (\log t)^{-2})'$, then $\g \in \RV_\B$ with
\[
	\B = \begin{pmatrix} 0 & 0 & \phantom{-}0 \\ 0 & 0 & -1 \\ 0 & 0 & \phantom{-}0 \end{pmatrix}.
\]
\item[(d)]
Define iterated logarithms by $\log_1 = \log$ and $\log_n = \log \circ \log_{n-1}$ for integer $n \geq 2$. Then the rate vector $\g = (\log_1, \ldots, \log_n)$ belongs to $\RV_{\0}$, where $\0$ stands for the $n$-by-$n$ zero matrix.
\end{flushenumerate}
\end{example}

\begin{example}[(powers of logarithms)]
\begin{flushenumerate}
\item[(a)]
For integer $k \geq 0$, put $\ell_k(x) = (\log x)^k / k!$. Then for integer $n \geq 1$, by \eqref{E:xJ}, we have
\[
	(\ell_n(x), \ldots, \ell_0(x))' = x^{\J_{n+1}} (0, \ldots, 0, 1)'.
\]
It follows that the $(n+1)$-dimensional rate vector $\g = (\ell_n, \ldots, \ell_0)'$ belongs to $\RV_{\J_{n+1}}$. Moreover, $\ell_n \in \Pi((\ell_{n-1}, \ldots, \ell_0)')$ with index $(1, 0, \ldots, 0)$. 

By linear transformation, if $p_1, p_2, \ldots, p_{n+1}$ are polynomials of degrees $n, n-1, \ldots, 0$, respectively, then the vector $\g$ given by $g_i(t) = p_i(\log t)$ ($i \in \{1, \ldots, n+1\}$) is regularly varying with index matrix $\Q \J_{n+1} \Q^{-1}$ for some upper triangular, invertible matrix $\Q$. In particular, if $p$ is a polynomial of degree $n$, then the function $p \circ \log$ belongs to $\Pi((\ell_{n-1}, \ldots, \ell_0)')$.
\item[(b)]
Let $f(t) = (\log t)^p$ with $p \in \RR \setminus \{0, 1, 2, \ldots\}$. For integer $n \geq 1$,
\[
	L_n(t) 
	= D^n (f \circ \exp) \circ \log(t) 
	= \frac{d^n}{dx^n} x^p \biggr|_{x = \log t}
	= (p)_n (\log t)^{p-n}.
\]
Since $(p)_n \neq 0$, Theorem~\ref{T:L} applies, and $f \in \Pi((L_1, \ldots, L_n)')$ with index $\bm{c} = (1, 0, \ldots, 0)$. The same result can be obtained with more effort via a Taylor expansion of $(1 + z)^p$ around $z \to 0$ in $f(tx) = \{\log(t)\}^p \{1 + \log(x)/\log(t)\}^p$.
\end{flushenumerate}
\end{example}

\begin{example}[(iterated logarithms)]
\label{Ex:loglog}
\begin{flushenumerate}
\item[(a)]
Let $f = \log \log$. For integer $n \geq 1$,
\[
	L_n(t)
	= D^n (f \circ \exp) \circ \log(t)
	= \frac{d^n}{dx^n} \log(x) \biggr|_{x = \log t}
	= (-1)^{n-1} (n-1)! (\log t)^{-n}.
\]
Theorem~\ref{T:L} applies, and $f \in \Pi((L_1, \ldots, L_n)')$ with index $\bm{c} = (1, 0, \ldots, 0)$.
\item[(b)]
Let $f(t) = (\log \log t)^2$. For integer $n \geq 1$,
\begin{align*}
	L_n(t) 
	&= D^n (f \circ \exp) \circ \log(t)
	= \frac{d^n}{dx^n} (\log x)^2 \biggr|_{x = \log t} \\
	&= 2 x^{-n} (a_n \log x + b_n) \biggr|_{x = \log t}
	= 2 (\log t)^{-n} (a_n \log \log t + b_n),
\end{align*}
where $a_n = (-1)^{n-1}$ and with $b_n$ recursively given by $b_1 = 0$ and $b_n = a_{n-1} - (n-1) b_{n-1}$ for integer $n \geq 2$. Since $a_n \neq 0$, Theorem~\ref{T:L} applies, and $f \in \Pi((L_1, \ldots, L_n)')$ with index $\bm{c} = (1, 0, \ldots, 0)$.
\item[(c)]
Let $f(t) = \log \log \log t$. For positive integer $n$, after some calculations,
\begin{align*}
	L_n(t) 
	&= D^n (f \circ \exp) \circ \log(t) = \frac{d^n}{dx^n} (\log \log x) \biggr|_{x = \log t} \\
	&= y^{-n} \sum_{j=1}^n a_{nj} (\log x)^{-j} \biggr|_{x = \log t}
	= (\log t)^{-n} \sum_{j=1}^n a_{nj} (\log \log t)^{-j},
\end{align*}
where the triangular array $\{a_{nj} : n = 1, 2, \dots; j = 1, \dots, n\}$ is recursively given by
\begin{align*}
	&a_{n1} = a_{nn} = (-1)^{n-1} (n-1)!, \\
	&a_{n+1,j} = (-n) a_{nj} - (j-1) a_{n,j-1}, \qquad j = 2, \ldots, n.
\end{align*}
Since $a_{n1} \neq 0$, Theorem~\ref{T:L} applies, and $f \in \Pi((L_1, \ldots, L_n)')$ with index $\bm{c} = (1, 0, \ldots, 0)$.
\end{flushenumerate}
\end{example}

\begin{example}[($\Pi$-varying functions with superlogarithmic growth)]
\begin{flushenumerate}
\item[(a)]
Let $f(t) = \exp \{ (\log t)^\alpha \}$ for some $\alpha \in (0, 1)$. We have $f \circ \exp(x) = \exp(x^\alpha)$, and it can be shown by induction that for integer $n \geq 1$,
\[
	\frac{d^n}{dx^n} \exp(x^\alpha) \sim (\alpha x^{\alpha-1})^n \exp(x^\alpha),
	\qquad x \to \infty.
\]
It follows that for integer $n \geq 1$,
\[
	L_n(t) 
	= D^n (f \circ \exp) \circ \log(t) 
	\sim \alpha^n (\log t)^{-n(1-\alpha)} \exp \{ (\log t)^\alpha \}.
\]
Since $L_n$ is slowly varying, Theorem~\ref{T:L} applies, and $f \in \Pi((L_1, \ldots, L_n)')$ with index $\bm{c} = (1, 0, \ldots, 0)$.
\item[(b)]
Let $f(t) = \exp ( \log t / \log \log t )$. It is not hard to see that $t Df(t) / f(t) = t D \log f(t) = o(1)$, so that $f$ is slowly varying, and thus also $Df \in \RV_{-1}$, implying $f \in \Pi$. Writing $u(x) = f \circ \exp(x) = \exp(x / \log x)$, one can show by induction that for integer $n \geq 1$,
\[
	D^n u(x) \sim u(x) (\log x)^{-n}.
\]
As a consequence, for integer $n \geq 1$,
\[
	L_n(t) = D^n (f \circ \exp) \circ \log(t)
	\sim f(t) \{ \log f(t) \}^{-n}
	\sim f(t) (\log \log t)^n (\log t)^{-n}.
\]
Since $L_n$ is slowly varying, Theorem~\ref{T:L} applies, and $f \in \Pi((L_1, \ldots, L_n)')$ with index $\bm{c} = (1, 0, \ldots, 0)$.
\end{flushenumerate}
\end{example}

\begin{example}[(asymptotically equivalent or not?)]
\label{Ex:log.loglog}
\begin{flushenumerate}
\item[(a)]
Let $f(t) = \log(t) \log \log(t)$. Since $f \circ \exp(x) = x \log(x)$ and $D (f \circ \exp) (x) = 1 + \log x$, we find that $L_n(t) = D^n (f \circ \exp) \circ \log(t)$ is given by
\begin{align}
\label{E:log.loglog:Ln:a}
	L_1(t) &= 1 + \log \log t, &
	L_n(t) &= (-1)^n (n-2)! \{ \log(t) \log \log(t) \}^{-n+1}, \qquad n \geq 2.
\end{align}
By Theorem~\ref{T:L}, $f \in \Pi((L_1, \ldots, L_n)')$ with index $\bm{c} = (1, 0, \ldots, 0)$.
\item[(b)]
Let $f$ be the inverse of the function $A(t) = \exp \{ t / \log(t) \}$. It is not hard to see that $f(t) \sim \log(t) \log \log(t)$, the function in item~(a). We have
\[
	q(t) 
	= \frac{1}{D \log A(t)} 
	= \frac{(\log t)^2}{\log t - 1} 
	= 1 + \log t + \frac{1}{\log t - 1}
	\sim \log t.
\]
By induction, one can show that for integer $n \geq 1$,
\[
	D^n q(t) \thicksim (-1)^{n-1} (n-1)! t^{-n}.
\]
Hence we are in the situation of Proposition~\ref{P:Djqi:2} ($\alpha = 0$). Putting $q_1 = q$ and $q_i = q Dq_{i-1}$ for $i \geq 2$, we find that, for $i \geq 2$,
\[
	q_i(t) 
	\sim (-1)^i (i-2)! t^{-i+2} \{q(t)\}^{i-1} t^{-1}
	\sim (-1)^i (i-2)! t^{-i+1} (\log t)^{i-1}.
\]
As a consequence, putting $L_n(t) = D^n (f \circ \exp) \circ \log(t) = q_n \circ f(t)$, we have
\begin{align}
\label{E:log.loglog:Ln:b}
	L_1(t) &\sim \log \log t, &
	L_n(t) &\sim (-1)^n (n-2)! (\log t)^{-n+1}, \qquad n \geq 2.
\end{align}
By Theorem~\ref{T:Gamma}, $f \in \Pi((L_1, \ldots, L_n)')$ with index $\c = (1, 0, \ldots, 0)$. Note that for $n \geq 2$, the functions $L_n$ in \eqref{E:log.loglog:Ln:a} and \eqref{E:log.loglog:Ln:b} are of different asymptotic orders.
\end{flushenumerate}
\end{example}

\begin{example}[(Lambert $W$ function)]
\label{Ex:LambertW}
The \emph{Lambert $W$ function} is defined by the identity $W(z) e^{W(z)} = z$, that is, $W$ is the inverse of the function $A(t) = t e^t$; see for instance \citet{Corless96}. (We obviously restrict attention to the positive branch of the function.) Note that
\begin{align*}
	q(t) &= \frac{A(t)}{DA(t)} = 1 - \frac{1}{1+t}, &
	D^n q(t) &= (-1)^{n-1} n! (1+t)^{-n-1}, \qquad n \geq 1.
\end{align*}
Proposition~\ref{P:Djqi:3} applies with $\alpha = -1$ and $q(\infty) = 1$. We find that for $n \geq 2$,
\[
	q_n(t) \sim (-2)_{n-2} t^{-n+2} (1+t)^{-2} \sim (-1)^n (n-1)! t^{-n}.
\]
By Theorem~\ref{T:Gamma}, $L_1(t) = t DW(t) = q(W(t)) = 1 + o(1)$ and for $n \geq 2$,
\begin{align*}
	L_n(t) 
	&= D^n(W \circ \exp) \circ \log(t) = q_n(W(t)) \\
	& \sim (-1)^n (n-1)! \{W(t)\}^{-n} \sim (-1)^n (n-1)! (\log t)^{-n},
\end{align*}
and $W \in \Pi((L_1, \ldots, L_n)')$ for every $n \geq 1$ with index $\bm{c} = (1, 0, \ldots, 0)$.
\end{example}

\begin{example}[(complementary Gamma function)] 
\label{Ex:compl.gamma}
For $b \in \RR$, $b \neq 0$, let the function $f$ be defined by the equation
\[
	\int_{f(t)}^\infty y^b e^{-y} \dy = \frac{1}{t}
\]
for large enough, positive $t$ (if $b = 0$ then $f = \log$). That is, $f$ is the inverse of the function $A$ given by
\begin{align*}
	A(x) &= \frac{1}{\Gamma(b + 1, x)}, &
	\Gamma(b+1, x) &= \int_x^\infty y^b e^{-y} \dy,
\end{align*}
the reciprocal of the complementary Gamma function. The function $q = A / DA$ is given by
\begin{align*}
	q(t) 
	&= \frac{A(t)}{DA(t)} = - \frac{1/A(t)}{D(1/A)(t)} \\
	&= t^{-b} e^t \int_t^\infty y^b e^{-y} \dy 
	= t \int_1^\infty z^b e^{t-tz} \dz 
	= t \int_0^\infty (1+v)^b e^{-tv} \dv \\
	&= 1 + b \int_0^\infty (1+v)^{b-1} e^{-tv} \dv.
\end{align*}
It follows that for integer $n \geq 1$,
\begin{multline}
\label{E:compl.gamma:q}
	D^n q(t) 
	= b \int_0^\infty (1+v)^{b-1} (-v)^n e^{-tv} \dv \\
	= b t^{-n-1} \int_0^\infty (1+u/t)^{b-1} (-u)^n e^{-u} \du 
	\sim (-1)^n n! b t^{-n-1}.
\end{multline}
Hence Proposition~\ref{P:Djqi:3} applies with $\alpha = -1$ and $q(\infty) = 1$, so that for integer $i \geq 2$,
\begin{align*}
	q_i(t) 
	&\sim (-2)_{i-2} t^{-i+2} Dq(t) \\
	&\sim - (-2)_{i-2} b t^{-i} 
	= (-1)^{i-1} (i-1)! b t^{-i}.
\end{align*}
Since $f(t) \sim \log(t)$, Theorem~\ref{T:Gamma} yields $L_1(t) = t Df(t) = q(f(t)) = 1 + o(1)$ while for $n \geq 2$,
\[
	L_n(t) = D^n (f \circ \exp) \circ \log(t) = q_n(f(t))
	\sim (-1)^{n-1} (n-1)! b (\log t)^{-n},
\]
and $f \in \Pi((L_1, \ldots, L_n)')$ with index $\c = (1, 0, \ldots, 0)$.
\end{example}

\begin{example}[(complementary error function)]
For $b \in \RR$ and $p \in (0, \infty)$, $p \neq 1$, let $f(t)$ be defined by the equation
\[
	\int_{f(t)}^\infty y^{b+p-1} e^{-y^p/p} \dy = \frac{1}{t}
\]
for sufficiently large, positive $t$. The function $f$ is the inverse of the function $A$ given by
\[
	A(x) = 1 \bigg/ \int_x^\infty y^{b+p-1} e^{-y^p/p} \dy, \qquad x > 0,
\]
the reciprocal of (a generalisation of) the complementary error function. It is not hard to see that $f(t) \sim (p \log t)^{1/p}$; by the way, in case $b = 0$ we have exactly $f(t) = (p \log t)^{1/p}$. We have
\[
	q(t) 
	= \frac{A(t)}{DA(t)} = - \frac{1/A(t)}{D(1/A)(t)} 
	= t^{-b-p+1} e^{t^p/p} \int_t^\infty y^{b+p-1} e^{-y^p/p} \dy.
\]
The function $q$ can be written as
\[
	q(t) = t^{1-p} \bar{q}(t^p/p), \qquad
	\bar{q}(x) = x^{-\bar{b}} e^x \int_x^\infty y^{\bar{b}} e^{-y} \dy, \qquad
	\bar{b} = b/p.
\]
The function $\bar{q}$ has the same structure as the function $q$ in Example~\ref{Ex:compl.gamma}. Using \eqref{E:compl.gamma:q} in that example together with Leibniz' product rule, one can show that for integer $n \geq 1$,
\begin{align*}
	q(t) &\sim t^{1-p}, &
	D^n q(t) &\sim D^n t^{1-p} = (1-p)_n t^{1-p-n}.
\end{align*}
By Proposition~\ref{P:Djqi:1} with $\alpha = 1-p$ and Proposition~\ref{P:Djqi:3} with $\alpha = 1-p$ and $q(\infty) = 0$, we find that for integer $n \geq 1$,
\[
	q_n(t) \sim t^{1 - np} \prod_{k=1}^{n-1} (1 - kp)
\]
and thus
\[
	L_n(t) = D^n(f \circ \exp) \circ \log(t) = q_n(f(t))
	\sim (p \log t)^{1/p-n} \prod_{k=1}^{n-1} (1 - kp).
\]
If $p$ is not of the form $1/k$ for some integer $k \in \{2, \ldots, n-1\}$, then $L_n$ is eventually of constant sign and $|L_n|$ is regularly varying, so that by Theorem~\ref{T:L}, $f \in \Pi((L_1, \ldots, L_n)')$ with index $\c = (1, 0, \ldots, 0)$.
\end{example}



\bibliographystyle{chicago}
\bibliography{Biblio}

\affiliationone{%
   Edward Omey\\
   Mathematics and Statistics, Hogeschool-Universiteit Brussel\\
	 Affiliated Researcher ETEW, Katholieke Universiteit Leuven\\
	 Stormstraat 2\\ 
	 B-1000 Brussels\\
	 Belgium\\
   \email{edward.omey@hubrussel.be}}
\affiliationtwo{%
   Johan Segers\\
   Institut de statistique, Universit\'e catholique de Louvain\\
	 Voie du Roman Pays 20\\ 
	 B-1348 Louvain-la-Neuve\\ 
	 Belgium
   \email{johan.segers@uclouvain.be}}
\end{document}